\documentclass[12pt]{article}
\usepackage[english]{babel}
\usepackage[square,numbers]{natbib}
\usepackage{amscd,float}
\usepackage{latexsym,amsmath,amssymb,amsbsy, amsthm}

\usepackage[colorlinks = TRUE, linkcolor = blue, citecolor = blue,
urlcolor = blue]{hyperref}
\usepackage{setspace}
\usepackage{lscape,fancyhdr,fancybox}
\usepackage{graphicx,epsfig, xy}
\usepackage{color}
\usepackage[hmarginratio=1:1, vmarginratio =5:5,
textheight=22cm,bindingoffset=1.5cm, textwidth=14.6cm]{geometry}
\usepackage{float}
\usepackage{graphicx}
\usepackage{caption}
\usepackage{subcaption}

\newcommand{\comment}[1]{}

\numberwithin{equation}{section}

\newtheorem{theorem}{Theorem}[section]
\newtheorem{lemma}{Lemma}[section]
\newtheorem{proposition}{Proposition}[section]

\newtheorem{prop}[theorem]{Proposition}

\numberwithin{equation}{section}

\newtheorem{remark}{Remark}[section]

\theoremstyle{definition}
\newtheorem{definition}{Definition}[section]

\newtheorem{conjecture}{Conjecture}[section]
\usepackage{amsfonts}

\usepackage{txfonts,enumerate}

\DeclareMathOperator{\E}{E} 
\DeclareMathOperator{\cov}{Cov}

\newcommand{\beq}{\begin{eqnarray}}
\newcommand{\eeq}{\end{eqnarray}}
\newcommand{\ben}{\begin{eqnarray*}}
\newcommand{\een}{\end{eqnarray*}}

\title{Contiguity and non-reconstruction results for planted partition models: the dense case}
 \author{
\sc Debapratim Banerjee
 \\ \small Dept. of Statistics\\ University of Pennsylvania\\ dban@wharton.upenn.edu\\}
\begin{document}
 \maketitle
\begin{abstract}
We consider the two block stochastic block model on $n$ nodes with asymptotically equal cluster sizes. The connection probabilities within and between cluster are denoted by $p_n:=\frac{a_n}{n}$ and $q_n:=\frac{b_n}{n}$ respectively. 
Mossel et al.\cite{MNS12} considered the case when $a_n=a$ and $b_n=b$ are fixed. They proved the probability models of the stochastic block model and that of Erd{\"o}s-R{\'e}nyi graph with same average degree are mutually contiguous whenever $(a-b)^2<2(a+b)$ and are asymptotically singular whenever $(a-b)^2>2(a+b)$. Mossel et al. \cite{MNS12} also proved that when $(a-b)^2<2(a+b)$ no algorithm is able to find an estimate of the labeling of the nodes which is positively correlated with the true labeling.
It is natural to ask what happens when $a_n$ and $b_n$ both grow to infinity. We prove that their results extend to the case when $a_n=o(n)$ and $b_n=o(n)$. We also consider the case when $\frac{a_n}{n} \to p \in (0,1)$ and $(a_n-b_n)^2= \Theta(a_n+b_n)$. Observe that in this case $\frac{b_n}{n} \to p$ also. We show that here the models are mutually contiguous if $(a_n-b_n)^2< 2(1-p)(a_n+b_n)$ and they are asymptotically singular if $(a_n-b_n)^2 > 2(1-p)(a_n+b_n)$. Further we also prove it is impossible find an estimate of the labeling of the nodes which is positively correlated with the true labeling whenever $(a_n-b_n)^2< 2(1-p)(a_n+b_n)$. The results of this paper justify the negative part of a conjecture made in Decelle et al.(2011) \cite{DKMZ11} for dense graphs.
\end{abstract} 
\section{Introduction}
In the last few years the stochastic block model has been one of the most active domains of modern research in statistics, computer science and many other related fields. In general a stochastic block model is a network with a hidden community structure where the nodes within the communities are expected to be connected in a different manner than the nodes between the communities. This model arises naturally in many problems of statistics, machine learning and data mining, but its applications further extends to
from population genetics \cite{JSD00} , where genetically similar
sub-populations are used as the clusters, to image processing \cite{SM00}, \cite{SHB99} , where the group of similar images acts as cluster, to the study of social networks ,
where groups of like-minded people act as clusters \cite{NWS02}. 

Recently a huge amount of effort has been dedicated to find out the clusters. Numerous different clustering algorithms have been proposed in literature. One might look at \cite{Jon67},\cite{Dempster77}, \cite{Bui1987}, \cite{Dyer1989}, \cite{R.B.B}, \cite{BC09}, \cite{Condon1999}, \cite{rohe2011}, \cite{McSherry} for some references.

One of the easiest examples of the stochastic block model is the planted partition model where one have only two clusters of more or less equal size. Formally,
\begin{definition}\label{def_planted}
For $n \in \mathbb{N}$, and $p,q \in [0,1]$ let $\mathcal{G}(n,p,q)$ denote the model of random,$\pm$ labelled graphs in which each
vertex $u$ is assigned (independently and uniformly at random) a label $\sigma_{u} \in \{ \pm 1\}$ and each edge between $u$ and $v$ are included independently with probability $p$ if they have the same label and with probability $q$ if they have different labels. 
\end{definition}
The case when $p$ and $q$ are sufficiently close to each other has got significant amount of interest in literature. Decelle et al. \cite{DKMZ11} made a fascinating conjecture in this regard. 
\begin{conjecture}\label{dec_con}
Let $p=\frac{a}{n}$ and $q=\frac{b}{n}$ where $a$ and $b$ are fixed real numbers. Then \\
i) If $(a-b)^2>2(a+b)$ then one can find almost surely a bisection of the vertices which is positively correlated with the original clusters.\\
ii) If $(a-b)^2<2(a+b)$ then the problem is not solveable.\\
iii) Further, there are no consistent estimators of $a$ and $b$ if $(a-b)^2<2(a+b)$ and there are consistent estimators of $a$ and $b$  whenever $(a-b)^2>2(a+b)$. 
\end{conjecture}
Coja-Oghlan \cite{Coja} solved part $i)$ of the problem when $(a-b)^2>C(a+b)$ for some large $C$ and finally
part $ii)$ and $iii)$ of Conjecture \ref{dec_con} was proved by Mossel et al. \cite{MNS12} and part $i)$ was solved by Mossel et al. \cite{MNS13} and  Massouli{\'e} \cite{Mas14} independently.  

Typically the problem is much more delicate when more than two communities are present in the sparse case. To keep things simple let us consider the general stochastic block model with $k$ asymptotically equal sized blocks with connection probabilities within and between blocks are given by $\frac{a}{n}$ and $\frac{b}{n}$ respectively. It was conjectured in Mossel et al \cite{MNS12} that for $k$ sufficiently large, there is a constant $c(k)$ such that whenever 
\[
c(k)< \frac{(a-b)^2}{a+(k-1)b} <k
\]
the reconstruction problem is solvable in exponential time, it is not solvable if $\frac{(a-b)^2}{a+(k-1)b}< c(k)$ and solvable in polynomial time if $k<\frac{(a-b)^2}{a+(k-1)b}$. The upper bound is known as  Kesten-Stigum threshold. Bordenave et al. \citep{Bor1} solved the reconstruction problem above a deterministic threshold by spectral analysis of non-backtraking matrix. One might look at Banks et al. \cite{bank16} for the non solvability part. They prove that the probability models of stochastic block model and that of Erd{\"o}s-R{\'e}nyi graph with same average degree are contiguous and the reconstruction problem is unsolvable if 
\[
d< \frac{2\log(k-1)}{k-1}\frac{1}{\lambda^2}.
\]
 Here $d= \frac{a+(k-1)b}{k}$ and $\lambda= \frac{a-b}{kd}$.
Abbe et al. \cite{AS16} provides an efficient algorithm for reconstruction above the Kesten-Stigum threshold.  Abbe et al. \cite{AS16} and Banks et al. \cite{bank16} also provide cases strictly below the  Kesten-Stigum threshold where the problem is solvable in exponential time.
  
 On the other hand, a different type of reconstruction problem was considered in  Mossel et al. \cite{MNS15} for denser graphs. They considered two different notions of recovery. The first one is weak consistency where one is interested in finding a bisection $\hat{\sigma}$ such that $\sigma$ and $\hat{\sigma}$ have correlation going to $1$ with high probability. The second one is called strong consistency. Here one is interested in finding a bisection $\hat{\sigma}$ such that $\hat{\sigma}$ is either $\sigma$ or $-\sigma$ with probability tending to $1$. Mossel et al. \cite{MNS15}  prove that weak recovery is possible if and only if $\frac{n(p_n-q_n)^2}{p_n+q_n} \to \infty$ and strong recovery is possible if and only if 
\[
\left(a_n+b_n-2\sqrt{a_nb_n}-1\right)\log n + \frac{1}{2} \log \log n \to \infty.
\]
Here $a_n= \frac{np_n}{\log n}$ and $b_n= \frac{nq_n}{\log n}$ respectively. Abbe et al. \citep{ABH14} studied the same problem independently in the logarithmic sparsity
regime. They prove that for $a=\frac{np_n}{\log n}$ and $b= \frac{nq_n}{\log n}$ fixed,  $(a + b)-2\sqrt{ab} > 1$ is sufficient for strong consistency and that $(a + b)- 2\sqrt{ab} \ge 1$
is necessary. We note that their results are implied by Mossel et al.\cite{MNS15}.

However, according to the best of our knowledge questions similar to part $ii)$ and $iii)$ of Conjecture \ref{dec_con} have not yet been addressed in dense case (i.e. when $a$ and $b$ increase to infinity) which is the main focus of this paper. 

Before stating our results we mention that the results in Mossel et al. \cite{MNS12} is more general than part $iii)$ of Conjecture \ref{dec_con}. Let $\mathbb{P}_n$ and $\mathbb{P}_n'$ be the sequence of probability measures induced by $\mathcal{G}(n,p,q)$ and $\mathcal{G}(n,\frac{p+q}{2},\frac{p+q}{2})$ respectively. Then \cite{MNS12} prove that whenever $a$ and $b$ are fixed numbers and $(a-b)^2<2(a+b)$, the measures $\mathbb{P}_n$ and $\mathbb{P}_n'$ are mutually contiguous i.e. for a sequence of events $A_n$, $\mathbb{P}_n(A_n) \to 0$ if and only if $\mathbb{P}_n'(A_n) \to 0$. Now part $iii)$ of Conjecture \ref{dec_con} directly follows from the contiguity. The proof in Mossel et al. \cite{MNS12} is based on calculating the limiting distribution of the short cycles and using a result of contiguity (Theorem 1 in Janson \cite{Jan} and Theorem 4.1 in Wormald \cite{Wormald}). However, one should note that the result from \cite{MNS12} doesn't directly generalize to the denser case. Since, one requires the limiting distributions of short cycles to be independent Poisson in order to use Janson's result. In our proof instead of considering the short cycles we consider the ``signed cycles"(to be defined later) which have asymptotic normal distributions. We also find a result analogous to Janson for the normal random variables in order to complete the proof.

On the other hand the original proof of non-reconstruction from Mossel et al. \cite{MNS12} relies on the coupling of $\mathbb{P}_n$ and $\mathbb{P}_n'$ with probability measure induced by Galton Watson trees of suitable parameters. However, it is well known that when the graph is sufficiently dense i.e. $a_n >> n^{o(1)}$ the coupling argument doesn't work. So our proof is based on fine analysis of some conditional probabilities. Technically, this proof is closely related to the non-reconstruction proof in section 6.2 of Banks et al. \cite{bank16} rather than the original proof given in Mossel et al. \cite{MNS12}. 

The paper is organized in the following manner. In Section \ref{sec:results} we build some preliminary notations and state our results. Section \ref{sec:contiguity} is dedicated for building a result analogous to Theorem 1 in Janson \cite{Jan}. In Section \ref{section:4} we define signed cycles and find their asymptotic distributions. Section \ref{sec:5} is dedicated to complete the proofs of our contiguity results. In Section \ref{sec:6} we prove the non-reconstruction result. Finally,  the paper concludes with an Appendix containing a proof of a result from random matrix theory used in this paper. 
\section{Our results}\label{sec:results}
Through out the paper a random graph will be denoted by $G$ and  $x_{i,j}$ will be used to denote the indicator random variable corresponding to an edge between the nodes $i$ and $j$. Further $\mathbb{P}_n$ and $\mathbb{P}_n'$ will be used to denote the sequence of probability measures induced by $\mathcal{G}(n,p_n,q_n)$ and $\mathcal{G}(n,\frac{p_n+q_n}{2},\frac{p_n+q_n}{2})$ respectively. For notational simplicity we denote $\frac{p_n+q_n}{2}$ by $\hat{p}_n$. 

Further, for any two labeling of the nodes $\sigma$ and $\tau$, we define their overlap to be 
\begin{equation}\label{overlap}
\mathrm{ov}(\sigma,\tau):= \frac{1}{n}\left(\sum_{i=1}^{n} \sigma_i\tau_i -\frac{1}{n} \left( \sum_{i=1}^{n}\sigma_i\right)\left(\sum_{i=1}^{n}\tau_i\right) \right).
\end{equation}
\noindent
We now state our results.
\begin{theorem}\label{thm:cont1} 
i)If $a_n,b_n\to \infty$, $a_n=o(n)$ and $(a_n-b_n)^2<2(a_n+b_n)$, then the probability measures $\mathbb{P}_n$ and $\mathbb{P}_n'$ are mutually contiguous. As a consequence, for any sequence of events $A_n$,  $\mathbb{P}_n(A_n) \to 0$ if and only if $\mathbb{P}_{n}'(A_n) \to 0.$ So there doesn't exists an estimator $(A_n,B_n)$ for $(a_n,b_n)$ such that $|A_n-a_n|+|B_n-b_n|=o_p(a_n-b_n)$.
\\
ii)If $a_n,b_n\to \infty$, $a_n=o(n)$ and $(a_n-b_n)^2>2(a_n+b_n)$, then the probability measures $\mathbb{P}_n$ and $\mathbb{P}_n'$ are asymptotically singular. Further there exists an estimator $(A_n,B_n)$ for $(a_n,b_n)$ such that $|A_n-a_n|+|B_n-b_n|=o_p(a_n-b_n)$.
\end{theorem}
\begin{theorem}\label{thm:cont2}
Suppose $\frac{a_n}{n} \to p \in (0,1)$  and let $c:= \frac{(a_n-b_n)^2}{(a_n+b_n)} \in (0,\infty)$, then the following are true:
\\
i) $\mathbb{P}_n$ and $\mathbb{P}_n'$ are mutually contiguous whenever ${\frac{c}{2(1-p)}}<1.$ So there doesn't exists an estimator $(A_n,B_n)$ for $(a_n,b_n)$ such that $|A_n-a_n|+|B_n-b_n|=o_p(a_n-b_n)$.
\\
ii) $\mathbb{P}_n$ and $\mathbb{P}_n'$  are asymptotically singular whenever ${\frac{c}{2(1-p)}}>1$. Further there exists an estimator $(A_n,B_n)$ for $(a_n,b_n)$ such that $|A_n-a_n|+|B_n-b_n|=o_p(a_n-b_n)$. 
\end{theorem}

\begin{theorem}\label{thm:nonrecon}
i) If $a_n,b_n \to \infty$, $a_n=o(n)$ and $(a_n-b_n)^2<2(a_n+b_n)$, then there is no reconstruction algorithm which performs better than the random guessing i.e. for any estimate of the labeling $\{\hat{\sigma}_{i}\}_{i=1}^{n}$ we have 
\begin{equation}\label{nonreconcond}
\mathrm{ov}(\sigma,\hat{\sigma}) \stackrel{P}{\to} 0.
\end{equation}
ii)Suppose $\frac{a_n}{n} \to p \in (0,1)$ and let $c:= \frac{(a_n-b_n)^2}{(a_n+b_n)} \in (0,\infty)$, then (\ref{nonreconcond}) holds when ${\frac{c}{2(1-p)}}<1.$ As a consequence, no reconstruction algorithm performs better than the random guessing.
\end{theorem}
\section{A result on contiguity}\label{sec:contiguity}
In this section we provide a very brief description of contiguity of probability measures. We suggest the reader to have a look at the discussion about contiguity of measures in Janson \cite{Jan} for further details. In this section we state several propositions and except for Proposition \ref{prop:norcont} and Proposition \ref{prop:wasser}, all the proofs can be found in Janson \cite{Jan}.  
\begin{definition}\label{def:cont}
Let $\mathbb{P}_n$
and $\mathbb{Q}_n$
be two sequences of probability measures, such that
for each $n$, $\mathbb{P}_n$ and $\mathbb{Q}_n$ both are defined on the same measurable space $(\Omega_n,\mathcal{F}_n)$. We then
say that the sequences are contiguous if for every sequence of measurable sets $A_n \subset \Omega_n$,
\[
\mathbb{P}_{n}(A_n) \to 0 \Leftrightarrow \mathbb{Q}_n(A_n) \to 0.
\]
\end{definition}
Definition \ref{def:cont} might appear a little abstract to some people.
However the following reformulation is perhaps more useful to understand the contiguity concept.
\begin{proposition}\label{prop:tough}
Two sequences of probability measures $\mathbb{P}_n$ and $\mathbb{Q}_{n}$ are contiguous if and only if for every $\varepsilon >0$ there exists $n(\varepsilon)$ and $K(\varepsilon)$ such that for all $n > n(\varepsilon)$ there exists a set $B_n \in \mathcal{F}_n$ with $\mathbb{P}_{n}(B_n^{c}), \mathbb{Q}_{n}(B_n^{c})\le \varepsilon$ such that 
\[
K(\varepsilon)^{-1}\le \frac{\mathbb{Q}_n(A_n)}{\mathbb{P}_n(A_n)}\le K(\varepsilon). ~~ \forall A_n \subset B_n.
\] 
\end{proposition}
Although Proposition \ref{prop:tough} gives an equivalent condition, verifying this condition is often difficult. However under the assumption of convergence of $\frac{d\mathbb{Q}_n}{d\mathbb{P}_n}$, one gets the following simplified result. 
\begin{proposition}\label{prop:useI}
Suppose that $L_n=\frac{d\mathbb{Q}_n}{d\mathbb{P}_n}$, regarded as a random variable on $(\Omega_n,\mathcal{F}_n,\mathbb{P}_n)$, converges in distribution to some random variable $L$ as $n \to \infty$. Then $\mathbb{P}_n$ and $\mathbb{Q}_n$
are contiguous if and only if $L > 0$ a.s. and $\E [L] = 1$.
\end{proposition}
We now introduce the concept of Wasserstein's metric which will be used in the proof of Proposition \ref{prop:norcont}. 
\begin{definition}\label{def:wass}
Let $F$ and $G$ be two distribution functions with finite $p$ th moment. Then the Wasserstein distance $W_p$ between $F$ and $G$ is defined to be 
\[
W_p(F,G)= \left[  \inf_{X \sim F, Y \sim G} \E|X-Y|^{p} \right]^{\frac{1}{p}}.
\]
Here $X$ and $Y$ are random variables having distribution functions $F$ and $G$ respectively.   
\end{definition} 
In particular, the following result will be useful in our proof:
\begin{proposition}\label{prop:wasser}
Suppose $F_n$ be a sequence of distribution functions and $F$ be a distribution function. Then $F_n$ converge to $F$ in distribution and $\int x^2 dF_n(x)\to \int x^2dF(x)$ if $W_2(F_n,F)\to 0$.  
\end{proposition}
The proof of Proposition \ref{prop:wasser} is well known. One might look at Mallows(1972)\cite{Mal72} for a reference.

With Proposition \ref{prop:useI} in hand, we now state the most important result in this section. This result will be used to prove Theorems \ref{thm:cont1} and \ref{thm:cont2}. Although, Proposition \ref{prop:norcont} is written in a complete different notation, one can check that it is analogous to Theorem 1 in Janson \cite{Jan}.
\begin{proposition}\label{prop:norcont}
Let $\mathbb{P}_n$ and $\mathbb{Q}_n$ be two sequences of probability measures such that for each $n$, both of them are defined on $(\Omega_n,\mathcal{F}_n)$. Suppose that for each $i \ge 3$, $X_{n,i}$ are random variables defined on $(\Omega_n,\mathcal{F}_n)$. Then the probability measures $\mathbb{P}_n$ and $\mathbb{Q}_n$ are mutually contiguous if the following conditions hold.
\begin{enumerate}[i)]
\item $\mathbb{P}_n<< \mathbb{Q}_n$ and $\mathbb{Q}_n<<\mathbb{P}_n $ for each $n$.
\item For each fixed $i \ge 3$, $X_{n,i}|\mathbb{P}_n \stackrel{d}{\to}Z_{i} \sim N(0,2i)$ jointly  and $X_{n,i}|\mathbb{Q}_n \stackrel{d}{\to}Z'_{i} \sim N(t^{\frac{i}{2}},2i)$ jointly such that $|t|<1$.  
\item $Z_{i}$ and $Z'_{i}$ are sequences of independent random variables.
\item \begin{equation}\label{eqn:limitsecond}
\E_{\mathbb{P}_n}\left[\left(\frac{d\mathbb{Q}_n}{d\mathbb{P}_n}\right)^2\right]\to \exp\left\{-\frac{t}{2}-\frac{t^2}{4}\right\}\frac{1}{\sqrt{1-t}}.
\end{equation}
\end{enumerate}
Further, \begin{equation}\label{eqn:definew}
\frac{d\mathbb{Q}_n}{d\mathbb{P}_n}|\mathbb{P}_n \stackrel{d}{\to} \exp\left\{ \sum_{i=3}^{\infty}\frac{2t^{\frac{i}{2}}Z_{i}-t^{i}}{4i} \right\}.
\end{equation}

\end{proposition}
\begin{proof}
In this proof for simplicity we denote $\frac{d\mathbb{Q}_n}{d\mathbb{P}_n}$ by $Y_n$.  We break the proof into two steps.\\
\textbf{Step 1.}
In this step we prove the random variable in R.S. of (\ref{eqn:definew}) is almost surely positive and $E[W]=1$.
Let us define 
\[
W=\exp\left\{ \sum_{i=3}^{\infty}\frac{2t^{\frac{i}{2}}Z_{i}-t^{i}}{4i} \right\}
\]
and 
\[
W^{(m)}= \exp\left\{ \sum_{i=3}^{m}\frac{2t^{\frac{i}{2}}Z_{i}-t^{i}}{4i} \right\}.
\]

As $Z_i \sim N(0,2i)$,
\[
\E\left[\exp\left\{\frac{2t^{\frac{i}{2}}Z_{i}-t^{i}}{4i} \right\}\right]=\exp\left\{ \frac{4t^{i}\times 2i}{2\times 16i^2}-\frac{t^i}{4i} \right\}=1.
\]
So $\{W^{(m)}\}_{m=3}^{\infty}$ is a martingale sequence and 
\[
\E\left[ W^{(m)2} \right]=\prod_{i=3}^{m} \exp\left\{ \frac{t^i}{2i} \right\}=\exp\left\{ \sum_{i=3}^{m} \frac{t^i}{2i} \right\}.
\]
Now 
\[
\sum_{i=3}^{\infty}\frac{t^i}{2i}= \frac{1}{2}\left(\log(1-t)-t -\frac{t^2}{2}\right) ~~ \forall ~~ |t|<1 .
\]
So $W^{(m)}$ is a $L^2$ bounded martingale. Hence, $W$ is a well defined random variable,
\[
\E[W^2]=\exp\left\{-\frac{t}{2}-\frac{t^2}{4}\right\}\frac{1}{\sqrt{1-t}}
\]
and $\E[W]=1$.

Now observe that $Z_i\stackrel{d}{=}-Z_i$ for each $i$ and whenever $|t|<1$, the series $\sum_{i=3}^{\infty}\frac{t^i}{4i}$ converges. So 
\[
W^{-1}\stackrel{d}{=} \exp\left\{\sum_{i=3}^{\infty}\frac{2t^{\frac{i}{2}}Z_{i}+t^{i}}{4i}\right\}. 
\]
However, $E[W^{-1}]= \exp\left\{ \sum_{i=3}^{\infty}\frac{t^i}{2i} \right\}< \infty$ implies $W>0$ a.s.
\\
\textbf{Step 2.}
Now we come to the harder task of proving $Y_n \stackrel{d}{\to} W$.
Since
 \[
 \limsup_{n \to \infty}\E_{\mathbb{P}_n}\left[\left(Y_n\right)^2\right]<\infty
 \]
 from condition $iv)$, the sequence $Y_n$ is tight. Hence from Prokhorov's theorem there is a sub sequence $\{ n_{k} \}_{k=1}^{\infty}$ such that $Y_{n_k}$ converge in distribution to some random variable $W(\{ n_{k} \})$. We shall prove that the distribution of $W(\{ n_{k} \})$ doesn't depend on the sub sequence $\{ n_{k} \}$. In particular, $W(\{ n_{k} \})\stackrel{d}{=} W$.
 \\
Since $Y_{n_k}$ converges in distribution to $W(\{ n_{k} \})$, for any further sub sequence $\{ n_{k_l} \}$ of $\{ n_{k}\}$, $Y_{n_{k_l}}$ also converges in distribution to $W(\{ n_{k} \})$.

Given $\varepsilon>0$ take $m$ big enough such that 
\[
\exp\left\{ \sum_{i=3}^{\infty} \frac{t^i}{2i} \right\}-\exp\left\{ \sum_{i=3}^{m} \frac{t^i}{2i} \right\} < \varepsilon.
\]
For this $m$, look at the joint distribution of 
$(Y_{n_{k}},X_{n_k,3},\ldots,X_{n_k,m})$. This sequence of $m-1$ dimensional random vectors with respect to $\mathbb{P}_{n_k}$ is also tight from condition $ii)$. So it has a further sub sequence such that $$(Y_{n_{k_l}},X_{n_{k_l},3},\ldots,X_{n_{k_l},m})|\mathbb{P}_{n_{k_l}}\stackrel{d}{\to}(H_1,\ldots,H_{m-1})\in (\Omega(\{ n_{k_l} \}),\mathcal{F}(\{ n_{k_l} \}),P(\{ n_{k_l} \}))(say).$$ 
Observe that the marginal distribution of $H_1$ is same as $W(\{ n_k \})$ and $(H_2,\ldots,H_{m-1})\stackrel{d}{=}(Z_3,\ldots,Z_m)$ from condition $ii)$. 

The most important part of this proof is to find suitable $\sigma$ algebras $\mathcal{F}_1 \subset\mathcal{F}_2 \in \mathcal{F}(\{ n_{k_l} \})$ and a random variable $V^{(m)}\stackrel{d}{=} W^{(m)}$ such that $(V^{(m)},\mathcal{F}_1)$ and $(H_1,\mathcal{F}_2)$ is a pair of martingales. 

From condition $iv)$ we have $ \limsup_{n \to \infty}\E_{\mathbb{P}_n}\left[ Y_n^2\right]< \infty$. As a consequence, the sequence the sequence $Y_{n_{k_l}}$ is uniformly integrable. This together with condition $i)$ will give us
\[
1=\E_{\mathbb{P}_{n_{k_l}}}\left[ Y_{n_{k_l}} \right] \to \E[H_1]=1.
\]
In other words, 
\begin{equation}\label{eqn_expder}
1=\int Y_{n_{k_l}} d\mathbb{P}_{n_{k_l}} \to \int H_1 dP(\{ n_{k_l} \})=1.
\end{equation}
Now take any positive bounded continuous function $f:\mathbb{R}^m \to \mathbb{R}$. By Fatou's lemma 
\begin{equation}\label{eqn_ineq}
\liminf \int f\left(X_{n_{k_l},3},\ldots,X_{n_{k_l},m}\right)Y_{n_{k_l}} d\mathbb{P}_{n_{k_l}} \ge \int f\left(H_2,\ldots,H_{m-1}\right)H_1dP(\{ n_{k_l} \}).
\end{equation}
However for any constant $\xi$ we have 
\[
\xi=\int \xi Y_{n_{k_l}} d\mathbb{P}_{n_{k_l}} \to \int \xi H_1 dP(\{ n_{k_l} \})=\xi
\]
from (\ref{eqn_expder}).

So (\ref{eqn_ineq}) holds for any bounded continuous function $f$. On the other hand replacing $f$ by $-f$ we have 
\begin{equation}\label{eqn_ineqII}
\lim \int f\left(X_{n_{k_l},3},\ldots,X_{n_{k_l},m}\right)Y_{n_{k_l}} d\mathbb{P}_{n_{k_l}} = \int f\left(H_2,\ldots,H_{m-1}\right)H_1dP(\{ n_{k_l} \}).
\end{equation}
Now applying condition $ii)$ we have  
\begin{equation}
\int f\left(X_{n_{k_l},3},\ldots,X_{n_{k_l},m}\right)Y_{n_{k_l}} d\mathbb{P}_{n_{k_l}}= \int f\left(X_{n_{k_l},3},\ldots,X_{n_{k_l},m}\right)d\mathbb{Q}_{n_{k_l}} \to \int f(H_2',\ldots,H_{m-1}') dQ.
\end{equation}
Here $(H_2',\ldots,H_{m-1}') \stackrel{d}{=}(Z_3',\ldots,Z_m')$ and $Q$ is the measure induced by $(H_2',\ldots,H_{m-1}')$. In particular, one can take the measure $Q$ such that it is defined on $(\Omega(\{ n_{k_l} \}),\mathcal{F}(\{ n_{k_l} \}))$ and $(H_2,\ldots, H_{m-1})$ themselves are distributed as $(H_2',\ldots,H_{m-1}')$ under the measure $Q$.  This is true due to the following observation. 
\[
\int f(H_2,\ldots, H_{m-1}) dQ= \int f(H_2,\ldots, H_{m-1}) V^{(m)}dP(\{ n_{k_l} \})
\]
for any bounded continuous function $f$. Here 
\[
V^{(m)}:= \exp\left\{ \sum_{i=2}^{m-1}\frac{2t^{\frac{i+1}{2}}H_{i}-t^{i+1}}{4(i+1)} \right\} \stackrel{d}{=} W^{(m)}.
\]

Since $f$ is any bounded continuous function, we have 
\[
\int_{A}  dQ= \int_{A}  V^{(m)}dP(\{ n_{k_l} \})
\] 
for any $A \in \sigma(H_2,\ldots, H_{m-1})$.

Now looking back into (\ref{eqn_ineqII}), we have 
\[
\int_{A}  V^{(m)}dP(\{ n_{k_l} \})=\int_{A} H_1dP(\{ n_{k_l} \}).
\]

 $V^{(m)}$ is $\sigma(H_2,\ldots, H_{m-1})$ measurable. So $(V^{(m)},\sigma(H_2,\ldots, H_{m-1}))$ and $(H_1,\sigma(H_1)\vee \sigma(H_2,\ldots, H_{m-1}))$ is a pair of martingales. 

 From Fatou's lemma 
\[
\E[H_1^2]\le \liminf_{n \to \infty} \E_{\mathbb{P}_n}[Y_n^2]= \exp\left\{ \sum_{i=3}^{\infty} \frac{t^i}{2i} \right\}.
\]
As a consequence, in the probability space $(\Omega(\{ n_{k_l} \}),\mathcal{F}(\{ n_{k_l} \}),P(\{ n_{k_l} \}))$,  we have 
\[
0 \le \E|H_1-V^{(m)}|^2 = \E[H_1^2]-\E[V^{(m)2}]< \varepsilon.
\]
So $W_2(F^{V^{(m)}},F^{H_1})< \sqrt{\varepsilon}$. Here $F^{V^{(m)}}$ and $F^{H_1}$ denote the distribution functions corresponding to $V^{(m)}$ and $H_1$ respectively. As a consequence, $W_2(F^{V^{(m)}},F^{H_1}) \to 0$ as $m \to \infty.$ Hence by Proposition \ref{prop:wasser}, $V^{(m)} \stackrel{d}{\to} H_1$. Using $W^{(m)}\stackrel{d}{=} V^{(m)}$, we get $W^{(m)} \stackrel{d}{\to} H_1$.

On the other hand, we have already proved $W^{(m)}$ converge to $W$ in $L^2$. So $H_1\stackrel{d}{=}W$. However, we also proved $H_1\stackrel{d}{=}W(\{ n_k\})$. Together, they imply $W(\{ n_k \})\stackrel{d}{=}W$ as required.
\end{proof}
\begin{remark}
One might observe that the second part in assumption $ii)$ of Proposition \ref{prop:norcont} is slightly weaker than (A2) in Theorem 1 of Janson \cite{Jan}. For our purpose this is sufficient since we use the fact that $Y_n=\frac{d\mathbb{Q}_n}{d\mathbb{P}_n}$. However, in Theorem 1 of Janson \cite{Jan} $Y_n$ can be any random variable.  
\end{remark}
\section{Signed cycles and their asymptotic distributions}\label{section:4}
We have discussed in the introduction that the proof of Mossel et al. \cite{MNS12} crucially used the fact that the asymptotic distribution of short cycles turn out to be Poisson. However, in the denser case one doesn't get a Poisson limit for the short cycles. So their proof doesn't work in the denser case. Here we consider instead the ``signed cycles" defined as follows:
\begin{definition}
For a random graph $G$ the signed cycle of length $k$ is defined to be:
\[
C_{n,k}(G)= \left(\frac{1}{\sqrt{np_{n,\mathrm{av}}(1-p_{n,\mathrm{av}})}}\right)^{k} \sum_{i_0,i_1,\ldots,i_{k-1}} (x_{i_0,i_1}-p_{n,\mathrm{av}})\ldots(x_{i_{k-1}i_0}-p_{n,\mathrm{av}})
\]
where $i_0,i_1,\ldots,i_{k-1}$ are all distinct and $p$ is the average connection probability i.e.\\
$p_{n,\mathrm{av}}= \frac{1}{n(n-1)} \sum_{i \neq j} \E[x_{i,j}].$ Observe that for $\mathcal{G}(n,p_n,q_n)$, $p_{n,\mathrm{av}}$ is equal to $\hat{p}_n$.
\end{definition}
One should note that when $k=3$ a similar kind of random variable was called ``signed triangle" in Bubeck et al. \cite{BDER14}
 
It is intuitive that one might expect an asymptotic normal distribution for  $C_{n,k}$'s when $n \to \infty$ and $\hat{p}_{n}$ is sufficiently large. Our next result is formalizing this intuition.
\begin{prop}\label{prop:signdistr}
i)When $G \sim \mathbb{P}_n'$, $n(p_n+q_n) \to \infty$ and $3\le k_1 < \ldots< k_l=o(\log(\hat{p}_nn))$, 
\begin{equation}\label{dist:null}
\left(\frac{C_{n,k_1}(G)}{\sqrt{2k_1}},\ldots, \frac{C_{n,k_l}(G)}{\sqrt{2k_l}}\right) \stackrel{d}{\to} N_l(0,I_l).
\end{equation}
ii) When $G \sim \mathbb{P}_n$, $np_n \to \infty$, $c=\frac{(a_n-b_n)^2}{(a_n+b_n)}=\Theta(1)$ and $3\le k_1 < \ldots< k_l=o\left(\min(\log(\hat{p}_nn),\sqrt{\log(n)})\right)$,
\begin{equation}\label{dist:alt}
\left(\frac{C_{n,k_1}(G)-\mu_1}{\sqrt{2k_1}},\ldots, \frac{C_{n,k_l}(G)-\mu_l}{\sqrt{2k_l}}\right) \stackrel{d}{\to}
N_l(0,I_l)
\end{equation}
where $\mu_i = \left(\sqrt{\frac{c}{2(1-\hat{p}_n)}}\right)^{k_i}$ for $1\le i \le m$.
\end{prop}

\noindent
The proof of Proposition \ref{prop:signdistr} is inspired from the remarkable paper by Anderson and Zeitouni \cite{AZ05}. However, the model in this case is simpler which makes the proof less cumbersome. The fundamental idea is to prove that the signed cycles converges in distribution by using the method of moments and the limiting random variables satisfy the Wick's formula. At first we state the method of moments.
\begin{lemma}\label{lem:mom}
Let $Y_{n,1},\ldots, Y_{n,l}$ be a random vector of $l$ dimension. Then $(Y_{n,1},\ldots, Y_{n,l}) \stackrel{d}{\to} (Z_1,\ldots,Z_{l})$ if the following conditions are satisfied:
\begin{enumerate}[i)]
\item 
\begin{equation}\label{eqn:momcond}
\lim_{n \to \infty}\E[X_{n,1}\ldots X_{n,m}] 
\end{equation}
exists for any fixed $m$ and $X_{n,i} \in \{ Y_{n,1},\ldots,Y_{n,l} \}$ for $1\le i \le m$.
\item(Carleman's Condition)\cite{Carl26}
\[
\sum_{h=1}^{\infty} \left(\lim_{n \to \infty}\E[X_{n,i}^{2h}]\right)^{-\frac{1}{2h}} =\infty ~~ \forall ~ 1\le i \le l.
\]
\end{enumerate} 
Further,  
\[
\lim_{n \to \infty}\E[X_{n,1}\ldots X_{n,m}]= \E[X_{1}\ldots X_{m}].
\]
Here $X_{n,i} \in \{ Y_{n,1},\ldots,Y_{n,l} \}$ for $1\le i \le m$ and $X_{i}$ is the in distribution limit of $X_{n,i}$.
\end{lemma} 
The method of moments is very well known and much useful in probability theory. We omit its proof.

\noindent
Now we stat the Wick's formula for Gaussian random variables which was first proved by Isserlis(1918)\cite{I18} and later on introduced by Wick\cite{W50} in the physics literature in 1950. 
\begin{lemma}(Wick's formula)\label{lem:wick}\cite{W50}
Let $(Y_1,\ldots, Y_{l})$ be a multivariate mean $0$ random vector of dimension $l$ with covariance matrix $\Sigma$(possibly singular). Then $((Y_1,\ldots, Y_{l}))$ is jointly Gaussian if and only if for any integer $m$ and $X_{i} \in \{ Y_1,\ldots,Y_{l} \}$ for $1\le i \le m$
\begin{equation}\label{eqn:wick}
\E[X_1\ldots X_{m}]=\left\{
\begin{array}{ll}
  \sum_{\eta} \prod_{i=1}^{\frac{m}{2}} \E[X_{\eta(i,1)}X_{\eta(i,2)}] & ~ \text{for $m$ even}\\
  0 & \text{for $m$ odd.}
\end{array}
\right.
\end{equation}
Here $\eta$ is a partition of $\{1,\ldots,m \}$ into $\frac{m}{2}$ blocks such that each block contains exactly $2$ elements and $\eta(i,j)$ denotes the $j$ th element of the $i$ th block of $\eta$ for $j=1,2$.
\end{lemma}
The proof of the aforesaid Lemma is omitted. However, it is good to note that the random variables $Y_1,\ldots, Y_{l}$ may also be the same. In particular, taking $Y_1=\cdots = Y_{l}$, Lemma \ref{lem:wick} also provides a description of the moments of Gaussian random variables.
\noindent
With Lemma \ref{lem:mom} and \ref{lem:wick} in hand, we now jump into the proof of Proposition \ref{prop:signdistr}.\\ 
\textbf{Proof of Proposition \ref{prop:signdistr}}\\
At first we introduce some notations and some terminologies. We denote an word $w$ to be an ordered sequence of integers (to be called letters) $(i_0,\ldots, i_{k-1},i_k)$ such that $i_0= i_k$ and all the numbers $i_j$  for $0\le j \le k-1$ are distinct. For a word $w=(i_0,\ldots, i_{k-1},i_k)$, its length $l(w)$ is $k+1$. The graph induced by an word $w$ is denoted by $G_w$ and defined as follows. One treats the letters $(i_0,\ldots,i_k)$ as nodes and put an edge between the nodes $(i_{j},i_{j+1})_{0\le j \le k-1}$. Note that for a word $w$ of length $k+1$, $G_w=(V_{w},E_{w})$ is just a $k$ cycle. For a word $w=(i_0,\ldots,i_{k})$ its mirror image is defined by $\tilde{w}= (i_0, i_{k-1},i_{k-2},\ldots, i_1,i_0)$. Further for a cyclic permutation $\tau$ of the set $\{0,1,\ldots,k-1 \}$, we define $w^{\tau}:=(i_{\tau(0)},\ldots, i_{\tau(k-1)},i_{\tau(0)})$. Finally two words $w$ and $x$ are  called paired if there is a cyclic permutation $\tau$ such that either $x^{\tau}=w$ or $\tilde{x}^{\tau}=w$.
 An ordered tuple of $m$ words, $(w_1,\ldots,w_m)$ will be called a sentence. For any sentence $a= (w_1,\ldots w_m)$, $G_{a}=(V_a,E_a)$ is the graph with $V_a=\cup_{i=1}^{m}V_{w_i}$ and $E_{a}= \cup_{i=1}^{m} E_{w_i}$.

\noindent  
\textbf{Proof of part i)}
We complete the proof of this part in two steps. In the first step the asymptotic variances of $(C_{n,k_1}(G),\ldots, C_{n,k_l}(G))$ will be calculated and the second step will be dedicated towards proving the asymptotic normality and independence of $(C_{n,k_1}(G),\ldots, C_{n,k_l}(G))$   .  

\noindent 
\textbf{Step 1:} Observe that when $G \sim \mathbb{P}_n'$ the distribution of $C_{n,k_1}(G),\ldots, C_{n,k_l}(G)$ is trivially independent of the labels $\sigma_i$ and $\E[C_{n,k}(G)]=0$ for any $k$. Since $\mathbb{P}_n'$ corresponds to the probability distribution induced by an Erd{\"o}s-R{\'e}nyi model. Now we prove that $\mathrm{Var}(C_{n,k}(G)) \sim 2k$ for any $k=o(\sqrt{n})$. Let for any word $w=(i_0,\ldots,i_k)$, $X_{w}:= \prod_{j=0}^{k-1}\left(x_{i_j,i_{j+1}}- \hat{p}_n\right).$ Now observe that 
\begin{equation}\label{eqn:varexpI}
\begin{split}
\mathrm{Var}(C_{n,k}) &= \left(\frac{1}{n\hat{p}_n(1-\hat{p}_n)}\right)^{k}\E\left[(\sum_{w} X_{w})^2\right]\\
&= \left(\frac{1}{n\hat{p}_n(1-\hat{p}_n)}\right)^{k} \E \left[ \sum_{w,x}X_{w}X_{x} \right].
\end{split}
\end{equation}
Since both $X_{w}$ and $X_x$ are product of independent mean $0$ random variables each coming exactly once, $\E[X_{w}X_{x}] \ne 0$ if and only if all the edges in $G_{w}$ are repeated in $G_{x}$. Observe that since $G_{w}$ and $G_{x}$ are cycles of length $k$, this is satisfied if and only if $w$ and $x$ are paired. There are $k$ many cyclic permutations $\tau$ of the set $\{0,\ldots,k-1 \}$ and for a given $w$ and $\tau$, there are only two possible choices of $x$ such that $w$ and $x$ are paired. These choices are obtained when $x^{\tau}=w$ and $\tilde{x}^{\tau}=w$. As a consequence for any word $w$, exactly $2k$ words are paired with it. Now observe that when $w$ and $x$ are paired,  $X_wX_x$ is a product of $k$ random variables each appearing exactly twice. As a consequence, $\E[X_{w}X_{x}]= \left(\hat{p}_n(1-\hat{p}_n)\right)^{k}.$ Also the total number of words is given by $n(n-1)\ldots (n-k+1)$ for the choices of $i_0,\ldots, i_{k-1}$. It is well known that 
\[
\frac{n(n-1)\ldots (n-k+1)}{n^{k}} \to 1
\]
whenever $k=o(\sqrt{n})$.
So 
\begin{equation}\label{eqn:varexpII}
\mathrm{Var}(C_{n,k})=  2k\left(\frac{1}{n\hat{p}_n(1-\hat{p}_n)}\right)^{k}n(n-1)\ldots (n-k+1)\left(\hat{p}_n(1-\hat{p}_n)\right)^{k} \sim 2k
\end{equation}
as long as $k=o(\sqrt{n}).$ This completes \textbf{Step 1} of the proof.

\noindent 
\textbf{Step 2:}  Now we claim that in order to complete \textbf{Step 2}, is enough to prove the following two limits.
\begin{equation}\label{eqn:lim1}
\lim_{n \to \infty} \E\left[C_{n,k_1}(G)C_{n,k_2}(G) \right] \to 0
\end{equation}
whenever $k_1 \ne k_2$ and there exists random variables $Z_1,\ldots,Z_l$ such that for any fixed $m$
\begin{equation}\label{eqn:lim2}
\lim_{n \to \infty}\E[X_{n,1}\ldots X_{n,m}]\to\left\{
\begin{array}{ll}
  \sum_{\eta} \prod_{i=1}^{\frac{m}{2}} \E[Z_{\eta(i,1)}Z_{\eta(i,2)}] & ~ \text{for $m$ even}\\
  0 & \text{for $m$ odd.}
\end{array}
\right.
\end{equation}
where $X_{n,i} \in \{ \frac{C_{n,k_1}(G)}{\sqrt{2k_1}},\ldots, \frac{C_{n,k_l}(G)}{\sqrt{2k_l}}\}$.

First observe that (\ref{eqn:lim2}) will simultaneously imply part $i)$ and $ii)$ of Lemma \ref{lem:mom}. Implication of $i)$ is obvious. However, for $ii)$ one can take $X_{n,i}$'s to be all equal and from Wick's formula (Lemma \ref{lem:wick}) the limiting distribution of $X_{n,i}$'s are normal. It is well known that normal random variables satisfy Carleman's condition. On the other hand (\ref{eqn:lim2}) also implies that the limit of $(\frac{C_{n,k_1}(G)}{\sqrt{2k_1}},\ldots, \frac{C_{n,k_l}(G)}{\sqrt{2k_l}})$ is jointly normal. Hence applying (\ref{eqn:lim1}), one gets the asymptotic independence. 

We first prove (\ref{eqn:lim1}). Observe that 
\[
\E\left[C_{n,k_1}(G)C_{n,k_2}(G) \right] = \left(\frac{1}{n\hat{p}_n(1-\hat{p}_n)}\right)^{\frac{k_1+k_2}{2}} \E \left[ \sum_{w,x}X_{w}X_{x} \right].
\]
However, here $l(w)=k_1+1$ and $l(x)=k_2+1$. So $\E \left[ \sum_{w,x}X_{w}X_{x} \right]=0$. As a consequence, (\ref{eqn:lim1}) holds.

Now we prove (\ref{eqn:lim2}). Let $l_i$ be the length of the signed cycle corresponding to $X_{n,i}$. Observe that $l_{i} \in \{k_1,\ldots, k_l\}$ for any $i$. At first we expand the L.S. of (\ref{eqn:lim2}).
\begin{equation}\label{eqn:break}
\E[X_{n,1}\ldots X_{n,m}]=\left(\frac{1}{n\hat{p}_n(1-\hat{p}_n)}\right)^{\frac{\sum_{i}l_i}{2}} \sum_{w_1,\ldots,w_m} E\left[X_{w_1}\ldots X_{w_m}\right].
\end{equation}
Here each of the graphs $G_{w_1},\ldots,G_{w_m}$ are cycles of length $l_1,\ldots,l_m$ respectively. So in order to have $ E\left[X_{w_1}\ldots X_{w_m}\right] \neq 0$, we need each of the edges in $G_{w_1},\ldots, G_{w_m}$ to be traversed more than once. The sentence $a:=(w_1,\ldots, w_m)$, formed by such $(w_1,\ldots,w_m)$ will be called a weak CLT sentence. Given a weak CLT sentence $a$, we introduce a partition $\eta(a)$, of $\{1,\ldots,m \}$ in the following way. If $i,j$ are in same block of the partition $\eta(a)$, then $G_{w_i}$ $G_{w_j}$ have at least one edge in common. 

As a consequence, we can further expand the L.S. of (\ref{eqn:break}) in the following way.
\begin{equation}\label{eqn:breakII}
\left(\frac{1}{n\hat{p}_n(1-\hat{p}_n)}\right)^{\frac{\sum_{i}l_i}{2}}\sum_{\eta}\sum_{w_1,\ldots,w_m~|~ \eta = \eta(w_1,\ldots,w_m)}E\left[X_{w_1}\ldots X_{w_m}\right]
\end{equation}

Observe that each block in $\eta$ should have at least $2$ elements. Otherwise, in this case $E\left[X_{w_1}\ldots X_{w_m}\right]=0$. As a consequence, the number of blocks in $\eta \le [\frac{m}{2}]$.

\noindent 
Now we prove that if the number of blocks in $\eta < [\frac{m}{2}]$, then 
\[
\left(\frac{1}{n\hat{p}_n(1-\hat{p}_n)}\right)^{\frac{\sum_{i}l_i}{2}}\sum_{\eta}\sum_{w_1,\ldots,w_m~|~ \eta = \eta(w_1,\ldots,w_m)}E\left[X_{w_1}\ldots X_{w_m}\right] \to 0.
\]
If $\eta(w_1,\ldots,w_m)$ have strictly less than $[\frac{m}{2}]$ blocks, then $a$ has strictly less than $[\frac{m}{2}]$ connected components. From Proposition 4.10 of Anderson and Zeitouni \cite{AZ05} it follows that in this case $\#V_{a}\le \sum_{i=1}^{m}\frac{l_i}{2}-1.$ However each connected component is formed by an union of several cycles so $V_a\le E_a$. Now the following lemma gives a bound on the number of weak CLT sentences having strictly less than $[\frac{m}{2}]$ connected components.
\begin{lemma}\label{lem:appendix}
Let $\mathcal{A}$ be the set of weak CLT sentences such that for each $a \in \mathcal{A}$, $\#V_{a}=t$. Then 
\begin{equation}\label{eqn:bddweakclt}
\# \mathcal{A} \le 2^{\sum_{i} l_i}\left(C_1\sum_{i}l_i\right)^{C_2m}\left(\sum_{i}l_i\right)^{3(\sum_{i}l_i-2t)}n^{t}.
\end{equation}
\end{lemma}
The proof of Lemma \ref{lem:appendix} is rather technical and requires some amount of random matrix theory. So we defer its proof to the appendix. However, assuming Lemma \ref{lem:appendix}, we have
\begin{equation}\label{eqn:bddweakcltII}
\begin{split}
&\left(\frac{1}{n\hat{p}_n(1-\hat{p}_n)}\right)^{\frac{\sum_{i}l_i}{2}}\sum_{a~:~ V_{a}\le \sum_{i=1}^{m}\frac{l_i}{2}-1} E\left[X_{w_1}\ldots X_{w_m}\right]\\
&\le \left(\frac{1}{n\hat{p}_n(1-\hat{p}_n)}\right)^{\frac{\sum_{i}l_i}{2}} \sum_{t=1}^{\sum_{i=1}^{m}\frac{l_i}{2}-1}\sum_{e = t}^{\sum_{i}\frac{l_i}{2}}2^{\sum_{i} l_i}\left(C_1\sum_{i}l_i\right)^{C_2m}\left(\sum_{i}l_i\right)^{3(\sum_{i}l_i-2t)}n^{t}\hat{p}_n^e.\\
&\le \left(\frac{1}{n\hat{p}_n(1-\hat{p}_n)}\right)^{\frac{\sum_{i}l_i}{2}} \sum_{t=1}^{\sum_{i=1}^{m}\frac{l_i}{2}-1}2^{\sum_{i} l_i}\left(C_1'\sum_{i}l_i\right)^{C_2'm}\left(\sum_{i}l_i\right)^{3(\sum_{i}l_i-2t)}n^{t}\hat{p}_n^t.\\
&\le \left(\frac{2}{\sqrt{(1-\hat{p}_n)}}\right)^{\sum_{i} l_i}\underbrace{\sum_{t=1}^{\sum_{i=1}^{m}\frac{l_i}{2}-1}\left( \frac{C_3(\sum_{i}l_i)^{C_4}}{n\hat{p}_n} \right)^{\sum_{i}\frac{l_i}{2}-t}}_{T_1(\text(say))}.
\end{split}
\end{equation}
where $C_3$ and $C_4$ are some known constants. The third in equality holds due to the following reason. As $\sum_{i=1}^{m}\frac{l_i}{2}-t \ge 1$,
\[
\left(C_1'\sum_{i}l_i\right)^{C_2'm}\left(\sum_{i}l_i\right)^{3(\sum_{i}l_i-2t)}=\left(C_1'(\sum_{i}l_i)^{\frac{C_2'm}{3(\sum_{i}l_i-2t)}+1}\right)^{3(\sum_{i}l_i-2t)}\le \left(C_1'(\sum_{i}l_i)^{\frac{C_2'm}{6}+1}\right)^{3(\sum_{i}l_i-2t)} .
\]

Observe that $T_1$ is just a geometric series. Further, lowest value of $\sum_{i=1}^{m}\frac{l_i}{2}-t$ is $1$. So we can give the following final bound to (\ref{eqn:bddweakcltII}),
\begin{equation}\label{eqn:bddweakcltIII}
\left(\frac{2}{\sqrt{(1-\hat{p}_n)}}\right)^{\sum_{i} l_i} C_5 \frac{C_3(\sum_{i}l_i)^{C_4}}{n\hat{p}_n}.
\end{equation}
where $C_5$ is another known constant.  When $k_l=o(\log(\hat{p}_n n))$ and $\sum_{i}l_i\le mk_l$   
\[
\left(\frac{2}{\sqrt{(1-\hat{p}_n)}}\right)^{mk_l}C_5\frac{C_3(mk_l)^{C_4}}{n\hat{p}_n}\to 0.
\]
Once this is proved all the other partitions left are pair partitions i.e. it has exactly $\frac{m}{2}$ many blocks. However, once such a partition $\eta$ is fixed then the choices within a block doesn't depend on the others. As a consequence, (\ref{eqn:wick}) is satisfied. This completes part i). \hfill $\square$

\noindent
\textbf{Proof of part ii)}
Let $d:= \frac{p_n-q_n}{2}$.
We have 
\begin{equation}
\begin{split}
&C_{n,k}(G)=\left(\frac{1}{n\hat{p}_n(1-\hat{p}_n)}\right)^{\frac{k}{2}} \sum_{i_0,i_1,\ldots,i_{k-1}} (x_{i_0,i_1}-\hat{p}_n)\ldots(x_{i_{k-1}i_0}-\hat{p}_n)\\
&= \left(\frac{1}{n\hat{p}_n(1-\hat{p}_n)}\right)^{\frac{k}{2}}\sum_{i_0,i_1,\ldots,i_{k-1}} (x_{i_0,i_1}-p_{i_0,i_1}+p_{i_0,i_1}-\hat{p}_n)\ldots(x_{i_{k-1}i_0}-p_{i_{k-1},i_{k}}+p_{i_{k-1},i_{k}}-\hat{p}_n)\\
&=\left(\frac{1}{n\hat{p}_n(1-\hat{p}_n)}\right)^{\frac{k}{2}}\sum_{i_0,i_1,\ldots,i_{k-1}}(x_{i_0,i_1}-p_{i_0,i_1}+{\sigma_{i_0}\sigma_{i_1}}d)\ldots(x_{i_{k-1}i_0}-p_{i_{k-1},i_{k}}+{\sigma_{i_{k-1}}\sigma_{i_k}}d)\\
&=\left(\frac{1}{n\hat{p}_n(1-\hat{p}_n)}\right)^{\frac{k}{2}}\left[\sum_{i_0,i_1,\ldots,i_{k-1}}(x_{i_0,i_1}-p_{i_0,i_1})\ldots (x_{i_{k-1}i_0}-p_{i_{k-1},i_{k}}) + d^{k}\prod_{j=0}^{k-1} {\sigma_{i_j}\sigma_{i_{j+1}}}\right]+ V_{n,k}.
\end{split}
\end{equation}
Here $p_{i,j}=p_n$ if $\sigma_{i}=\sigma_{j}$ and $q_n$ otherwise. 
\\
 At first we prove that 
\begin{equation}\label{eqn:prodassign}
\prod_{j=0}^{k-1} {\sigma_{i_j}\sigma_{i_{j+1}}}=1
\end{equation}
irrespective of the values of $\sigma_{i}$'s. To prove this, without loss of generality let us assume $\sigma_{i_0}=+1.$ We now look at the runs of $+1$'s and $-1$'s in $\sigma_{i_j}$'s. Since $i_0=i_k$, the value of $\sigma_{i_k}$ is also $1$. So the any such assignment of $\sigma$ start with a run of $+1$ and end with a run of $+1$. Also, the runs of $+1$'s and $-1$ alternate. Hence there is only even number of change of signs in the whole assignment. Now 
\[
\prod_{j=0}^{s} {\sigma_{i_j}\sigma_{i_{j+1}}} = - \prod_{j=0}^{s+1} {\sigma_{i_j}\sigma_{i_{j+1}}}
\]
if and only if $\sigma_{i_s}=-\sigma_{i_{s+1}}$. This completes the proof of (\ref{eqn:prodassign}). 

The proof of asymptotic normality and independence of
\[
D_{n,k}(G):=\left(\frac{1}{n\hat{p}_n(1-\hat{p}_n)}\right)^{\frac{k}{2}}\left[\sum_{i_0,i_1,\ldots,i_{k-1}}(x_{i_0,i_1}-p_{i_0,i_1})\ldots (x_{i_{k-1}i_0}-p_{i_{k-1},i_{k}})\right]
\]
is exactly same as part i). We only note that here the variance is also $2k$. To see this, at first observe that 
\[
d= \sqrt{\frac{c\hat{p}_n}{2n}}
\]
and whenever, $k=o(\log(\hat{p}_nn))$ both 
\begin{equation}\label{lim1}
\lim_{n \to \infty}\left(\frac{(\hat{p}_n+d)(1-\hat{p}_n-d)}{\hat{p}_n(1-\hat{p}_n)}\right)^{\frac{k}{2}} =1
\end{equation}
and 
\begin{equation}\label{lim2}
\lim_{n \to \infty}\left(\frac{(\hat{p}_n-d)(1-\hat{p}_n+d)}{\hat{p}_n(1-\hat{p}_n)}\right)^{\frac{k}{2}} =1.
\end{equation}
It is easy to see that $\mathrm{Var}\left(\frac{D_{n,k}(G)}{\sqrt{2k}}\right)$ lies between L.S. of ( \ref{lim1}) and (\ref{lim2}). As a consequence, $\mathrm{Var}\left(\frac{D_{n,k}(G)}{\sqrt{2k}}\right) \to 1$. 
\\
Now our final task is to prove $\mathrm{Var}(V_{n,k})\to 0$.
\\
Let us fix a word $w$ and let $E_{f}\subset E_{w}$ be any subset. Then 
\[
V_{n,k}= \sum_{w} V_{n,k,w}
\]
where 
\[
V_{n,k,w}:= \left(\frac{1}{n\hat{p}_n(1-\hat{p}_n)}\right)^{\frac{k}{2}}\sum_{E_{f}\subset E_{w}} \prod_{e \in E_{f}}\sigma_{e}d\prod_{e \in E \backslash E_{f}}(x_{e}-p_{e}).
\]
Here for any edge ${i,j}$, $x_{e}=x_{i,j}$, $p_{e}=p_{i,j}$ and $\sigma_{e}={\sigma_{i}\sigma_{j}}$.  
Now 
\[
\mathrm{Var}(V_{n,k})=\sum_{w,x}\cov(V_{n,k,w}, V_{n,k,x}).
\]
We now find an upper bound of $\cov(V_{n,k,w}, V_{n,k,x})$. 
\\
At first fix any word $w$ and the set $E_{f} \subset E_{w}$ and consider all the words $x$ such that $E_{{w}}\cap E_{x}= E_{w} \backslash E_{f}$. As every edge in $G_{w}$ and $G_{x}$ appear exactly once,
\begin{equation}
\begin{split}
&\cov(V_{n,k,w}, V_{n,k,x})= \sum_{E_w\backslash E' \subset E_{w} \backslash E_{f}}\left(\frac{1}{n\hat{p}_n(1-\hat{p}_n)}\right)^{k}\prod_{e \in E'}(\pm d^2)\E\prod_{e \in E_{w}\backslash E'}(x_{e}-p_{e})^{2}\\
&=\sum_{E_w\backslash E' \subset E_{w} \backslash E_{f}}\left(\frac{1}{n\hat{p}_n(1-\hat{p}_n)}\right)^{k} 
\pm d^{2\#E'} (1+o(1))\left(\hat{p}_n(1-\hat{p}_n)\right)^{k-\#E'}\\
&\le\sum_{E_w\backslash E' \subset E_{w} \backslash E_{f}}(1+o(1))\left(\frac{1}{n\hat{p}_n(1-\hat{p}_n)}\right)^{k}\left( \frac{c}{2} \right)^{\#E'}\left(\frac{\hat{p}_n}{n}\right)^{\#E'}\hat{p}_n^{k-\#E'}\\
& \le (C)^{k}\frac{1}{n^{k+\#E_{f}}}
\end{split}
\end{equation}
where $C$ is some known constant.
The last inequality holds since $\#E'\ge \#E_{f}$ and  $\#(E_w\backslash E' \subset E_{w} \backslash E_{f})\le 2^{k}$.

Observe that the graph corresponding to the edges $E_{w} \backslash E_{f}$ is a disjoint collection of straight lines. Let the number of such straight lines be $\zeta$. Obviously $\zeta \le \#(E_{w} \backslash E_{f})$.  The number of ways these $\zeta$ components can be placed in $x$ is bounded by $k^{\zeta}\le k^{\#(E_{w} \backslash E_{f})}$ and all other nodes in $x$ can be chosen freely. So there is at most $n^{k-\#V_{E_{w}\backslash E_f}}k^{\#(E_{w} \backslash E_{f})}$ choices of such $x$. Here $V_{E_{w}\backslash E_{f}}$ is the set of vertices of the graph corresponding to $(E_{w} \backslash E_{f})$. Observe that, whenever $\#E_{f}>0$, $E_{w}\backslash E_{f}$ is a forest so $$\#V_{E_{w}\backslash E_{f}}\ge \#(E_{w} \backslash E_{f})+1 \Leftrightarrow k-\#V_{E_{w}\backslash E_f} \le \#E_{f}-1.$$

As a consequence, 
\begin{equation}\label{eqn:booboo}
\sum_{x~|~ E_{{w}}\cap E_{x}= E_{w} \backslash E_{f}}\cov(V_{n,k,w}, V_{n,k,x})\le (C)^{k}\frac{1}{n^{k+\#E_{f}}}n^{E_{f}-1}k^{\#(E_{w} \backslash E_{f})}\le  (C)^{k} \frac{1}{n^{k+1}}k^{k} .
\end{equation}
R.S. of (\ref{eqn:booboo}) doesn't depend on $E_{f}$
and there are at most $2^{k}$ nonempty subsets $E_{f}$ of $E^{w}$.
So
\[
\sum_{x}\cov(V_{n,k,w}, V_{n,k,x}) \le (2C)^{k} k^{k}\frac{1}{n^{k+1}}.
\] 
Finally there are at most $n^{k}$ many $w$. So
\begin{equation}\label{eqn:final}
\sum_{w}\sum_{x}\cov(V_{n,k,w}, V_{n,k,x})\le (2C)^{k}k^{k}\frac{1}{n}.
\end{equation}
Now we use the fact $k=o(\sqrt{\log(n)})$. In this case 
\[
k\log(2C)+k\log(k)\le \sqrt{log(n)}\log(\sqrt{\log n}) = o(log(n)) \Leftrightarrow (2C)^{k}k^{k}=o(n).
\]
This concludes the proof. 
\hfill $\square$
\section{Calculation of second moment and completion of the proofs of Theorems \ref{thm:cont1} and \ref{thm:cont2}}\label{sec:5}
With Propositions \ref{prop:norcont} and \ref{prop:signdistr} in hand the rest of the proof of Theorems \ref{thm:cont1} and \ref{thm:cont2} should be very straight forward.
We at first prove that $\lim_{n \to \infty}\mathbb{E}\left(\frac{d\mathbb{P}_n}{d\mathbb{P}_n'}\right)^2$ is r.s. of (\ref{eqn:limitsecond}) with $t= \frac{c}{2}$ and $\frac{c}{2(1-p)}$ whenever $a_n=o(n)$ and $\frac{a_n}{n} \to p$ respectively.
\begin{lemma}\label{lem:secondmoment}
Let $Y_n:= \frac{d\mathbb{P}_n}{d\mathbb{P}_n'}$. Then the following are true 
\begin{enumerate}[i)]
\item When $p_n \to 0$(i.e. $a_n=o(n)$), $$\E_{\mathbb{P}_n'}[Y_n^2]\to \exp\left\{-\frac{t}{2}-\frac{t^2}{4}\right\}\frac{1}{\sqrt{1-t}},~~ t=\frac{c}{2}<1. $$ 
\item When $p_n \to p \in (0,1)$ 
$$\E_{\mathbb{P}_n'}[Y_n^2]\to \exp\left\{-\frac{t}{2}-\frac{t^2}{4}\right\}\frac{1}{\sqrt{1-t}},  ~~ t= \frac{c}{2(1-p)}<1. $$ 
\end{enumerate}
\end{lemma}
\begin{proof}
The proof of Lemma \ref{lem:secondmoment} is similar to the proof of Lemma 5.4. in Mossel et al. \cite{MNS12}.  We only provide a proof of part $ii)$.  The proof of part $i)$ is similar. The notations used in this proof are slightly different from that of Lemma 5.4. in Mossel et al. \cite{MNS12} for understanding part $ii)$ better. 
\\
At first we introduce some notations. Given a labeled graph $(G,\mathbf{\sigma})$ we define 
\begin{equation}\label{defn:W}
W_{uv}= W_{uv}(G,\mathbf{\sigma})=\left\{
\begin{array}{ll}
\frac{p_n}{\hat{p}_n} & \text{if $\sigma_u\sigma_v=1$ and $(u,v)\in E$}\\
\frac{q_n}{\hat{p}_n} & \text{if $\sigma_u\sigma_v=-1$ and $(u,v) \in E$}\\
\frac{1-p_n}{1-\hat{p}_n} & \text{if $\sigma_u\sigma_v=1$ and $(u,v)\notin E$}\\
\frac{1-q_n}{1-\hat{p}_n} & \text{if $\sigma_u\sigma_v=-1$ and $(u,v)\notin E$}
\end{array}
\right.
\end{equation}
and define $V_{uv}$ by the same formula, but with $\sigma$ replaced by $\tau$.
Now 
 \[
 Y_n= \frac{1}{2^n} \sum_{\sigma \in \{1,-1\}^{n}} \prod_{(u,v)} W_{uv}
 \]
 and 
 \[
 Y_n^2= \frac{1}{2^{2n}} \sum_{\sigma,\tau} \prod_{(u,v)} W_{uv}V_{uv}.
 \]
 Since $\{W_{uv}\}$ are independent given $\sigma$, it follows that
 \[
 \E_{\mathbb{P}_n'}(Y_n^2)= \frac{1}{2^{2n}} \sum_{\sigma,\tau}\prod_{(u,v)} \E_{\mathbb{P}_n'}\left(W_{uv}V_{uv}\right).
 \]
 Now we consider the following cases:
 \begin{enumerate}
 \item $\sigma_u\sigma_v=1$ and  $\tau_u\tau_v=1$.
 \item $\sigma_u\sigma_v=-1$ and $\tau_u\tau_v=-1$.
 \item $\sigma_u\sigma_v=1$ and $\tau_u\tau_v=-1$.
 \item $\sigma_u\sigma_v=-1$ and $\tau_u\tau_v=1$.
 \end{enumerate}
Let $t= \frac{c}{2(1-p)}$. We at first calculate $\E_{\mathbb{P}_n'}(W_{uv}V_{uv})$ for cases $1$ and $3$. 
 \\
 \textbf{Case 1:}
 \begin{equation}\label{case1}
 \begin{split}
 \E_{\mathbb{P}_n'}(W_{uv}V_{uv}) &= \left( \frac{p_n}{\hat{p}_n} \right)^2 \hat{p}_n+ \left( \frac{1-p_n}{1-\hat{p}_n} \right)^2(1-\hat{p}_n).\\
 & = \frac{p^2_n}{\hat{p}_n} + \frac{(1-p_n)^2}{1-\hat{p}_n}\\
 &= \frac{(\hat{p}_n+d_n)^2}{\hat{p}_n}+ \frac{(1-\hat{p}_n-d_n)^2}{1-\hat{p}_n}\\
 &= 1+ d_n^2(\frac{1}{\hat{p}_n}+\frac{1}{1-\hat{p}_n})= 1+ \frac{d_n^2}{\hat{p}_n(1-\hat{p}_n)}=1+ \frac{c}{2n(1-\hat{p}_n)}\\
 &= 1+ \frac{t_n}{n}
 \end{split}
 \end{equation}
 where $d_n= \frac{p_n-q_n}{2}$ and $t_n:=\frac{c}{2(1-\hat{p}_n)}=(1+o(1))t$.
 \\
 \textbf{Case 3:}
  \begin{equation}\label{case3}
 \begin{split}
 \E_{\mathbb{P}_n'}(W_{uv}V_{uv}) &= \left( \frac{p_n}{\hat{p}_n}\cdot\frac{q_n}{\hat{p}_n} \right) \hat{p}_n+ \left( \frac{1-p_n}{1-\hat{p}_n}\cdot\frac{1-q_n}{1-\hat{p}_n} \right)(1-\hat{p}_n).\\
 & = \frac{p_nq_n}{\hat{p}_n} + \frac{(1-p_n)(1-q_n)}{1-\hat{p}_n}\\
 &= \frac{(\hat{p}_n+d_n)(\hat{p}_n-d_n)}{\hat{p}_n}+ \frac{(1-\hat{p}_n-d_n)(1-\hat{p}_n+d_n)}{1-\hat{p}_n}\\
 &= 1- d_n^2(\frac{1}{\hat{p}_n}+\frac{1}{1-\hat{p}_n})= 1- \frac{d_n^2}{\hat{p}_n(1-\hat{p}_n)}=1-\frac{t_n}{n}
 \end{split}
 \end{equation}
 It is easy to observe that $ \E_{\mathbb{P}_n'}(W_{uv}V_{uv})= 1+\frac{t_n}{n}$ and $1-\frac{t_n}{n}$ for Case $2$ and Case $4$ respectively.
 
 \noindent
 We now introduce another parameter $\rho=\rho(\sigma,\tau)=\frac{1}{n}\sum_{i}\sigma_i\tau_i.$ Let $S_{\pm}$ be the number of $\{ u,v \}$ such that $\sigma_u\sigma_v\tau_u\tau_v=\pm 1$ respectively. It is easy to observe that 
 \begin{equation}
 \rho^2= \frac{1}{n}+ \frac{2}{n^2}(S_{+}-S_{-})
 \end{equation}
 and 
 \begin{equation}
 1-\frac{1}{n}= \frac{2}{n^2}(S_{+}+S_{-}).
 \end{equation}
 So 
 \begin{equation}
 S_{+}= (1+\rho^2)\frac{n^2}{4}- \frac{n}{2}, ~~ S_{-}= (1-\rho^2)\frac{n^2}{4}.
 \end{equation}
 
 \noindent 
 Now 
 \begin{equation}\label{expy^2}
 \begin{split}
 \E_{\mathbb{P}_n'}(Y_n^2)&= \frac{1}{2^{2n}}\sum_{\sigma,\tau}\left(1+\frac{t_n}{n}\right)^{S_{+}}\left(1-\frac{t_n}{n}\right)^{S_{-}}\\
 &= \frac{1}{2^{2n}} \sum_{\sigma,\tau}\left(1+\frac{t_n}{n}\right)^{(1+\rho^2)\frac{n^2}{4}- \frac{n}{2}}\left(1-\frac{t_n}{n}\right)^{(1-\rho^2)\frac{n^2}{4}}.
 \end{split}
 \end{equation}
 Observe that $t_n=(1+o(1))t$ is a bounded sequence. It is easy to check by taking logarithm and Taylor expansion that for any bounded sequence $x_n$, 
 \[
 \left(1+\frac{x_n}{n}\right)^{n^2}=(1+o(1))\exp\left\{ nx_n-\frac{1}{2}x_n^2\right\}.  
 \]
 So we can write R.S. of (\ref{expy^2}) as 
 \begin{equation}\label{final}
 \begin{split}
 &(1+o(1))\frac{1}{2^{2n}}\sum_{\sigma,\tau}e^{-\frac{t_n}{2}}\exp\left[\left(nt_n-\frac{t_n^2}{2}\right)\left(\frac{1+\rho^2}{4}\right)\right]\times \exp\left[\left(-nt_n-\frac{t_n^2}{2}\right)\left(\frac{1-\rho^2}{4}\right)\right]\\
= &(1+o(1))\frac{1}{2^{2n}}\sum_{\sigma,\tau}e^{-\frac{t_n}{2}-\frac{t_n^2}{4}} \exp\left[ \frac{nt_n\rho^2}{2} \right]\\
=&(1+o(1))e^{-\frac{t_n}{2}-\frac{t_n^2}{4}}\frac{1}{2^{2n}}\sum_{\sigma,\tau} \exp\left[ \frac{(1+o(1))tn\rho^2}{2} \right]
 \end{split}
 \end{equation}
 From Lemma 5.5 in Mossel et al. \cite{MNS12} 
 \[
 \frac{1}{2^{2n}}\sum_{\sigma,\tau}\exp\left[ \frac{(1+o(1))nt\rho^2}{2} \right] \to \frac{1}{\sqrt{1-t}}.
 \]
 So R.S. of (\ref{final}) converges to 
 \[
 \exp\left\{-\frac{t}{2}-\frac{t^2}{4}\right\}\frac{1}{\sqrt{1-t}}
 \]
 as required.
\end{proof}

\noindent
\textbf{Proof of Theorem \ref{thm:cont1} and \ref{thm:cont2}:} 
The proofs of Theorem \ref{thm:cont1} and \ref{thm:cont2} only differ in the value of $t$. For the case $a_n=o(n)$, $t= \frac{c}{2}$ and $t=\frac{c}{2(1-\hat{p})}$ for the other case.  We prove only Theorem \ref{thm:cont1}. Proof of Theorem \ref{thm:cont2} is similar after plugging in the appropriate value of $t$.
\\
\textbf{Proof of part i)} We take $X_{n,i}=C_{n,i}(G)$. 

\noindent
At first observe that when $a_n=o(n)$(i.e. $p_n,q_n\to 0$) for any fixed $i$, $\mu_i:= \left(\sqrt{\frac{c}{2(1-\hat{p}_n)}}\right)^{i}$ converges to $\left(\frac{c}{2}\right)^{\frac{i}{2}}$ as $n \to \infty$. 

From Proposition \ref{prop:signdistr} and Lemma \ref{lem:mom} we see that $C_{n,i}(G)$'s satisfy all the required conditions for Proposition \ref{prop:norcont}. Hence $\mathbb{P}_n$ and $\mathbb{P}_n'$ are mutually contiguous. 

It is easy to see that the average degree $\hat{d}_n:= \frac{1}{n}\sum_{i \neq j}x_{i,j}$ has mean $\frac{a_n+b_n}{2}$ and variance $O(\frac{a_n+b_n}{n})$. So $$\hat{d}_n-\frac{a_n+b_n}{2}=o_p(\sqrt{a_n+b_n})=o_p(a_n-b_n)$$

Suppose under $\mathbb{P}_n$ there exist estimators $A_n$ of $a_n$ and $B_n$ of $b_n$ such that 
\[
|A_n-a_n|+|B_n-b_n|=o_p(a_n-b_n).
\]
Then $2(\hat{d}_n-B_n)-(a_n-b_n)=o_p(a_n-b_n)$ i.e. 
\[
\frac{2(\hat{d}_n-B_n)}{a_n-b_n}|\mathbb{P}_n \stackrel{P}{\to} 1.
\]
However, from the fact that $\mathbb{P}_n$ and $\mathbb{P}_n'$ are contiguous we also have
\[
\frac{2(\hat{d}_n-B_n)}{a_n-b_n}|\mathbb{P}_n' \stackrel{P}{\to} 1
\]
which is impossible. \\
\textbf{Proof of part ii)} It is easy to observe that $\mathbb{P}_n$ and $\mathbb{P}_n'$ are asymptotically singular as for any $k_n \to \infty$, $\frac{\mu_{k_n}}{\sqrt{2k_n}} \to \infty.$ Now we construct estimators for $a_n$ and $b_n$. Let us define 
\[
\hat{f}_{n,k_n}=\left\{  
\begin{array}{ll}
\left(\sqrt{2k_n} C_{n,k_n}(G)\right)^{\frac{1}{k_n}} & \text{if} ~~ C_{n,k_n}(G)>0\\
0& \text{otherwise}.
\end{array}
\right.
\]
It is easy to see that under $\mathbb{P}_n$ $\hat{f}_{n,k_n}\stackrel{P}{\to} \frac{a_n-b_n}{\sqrt{2(a_n+b_n)}}=\sqrt{\frac{c}{2}}$ as $k_n \to \infty$. We have seen earlier that under $\mathbb{P}_n$
\begin{equation}
\begin{split}
&\frac{\hat{d}_n-\frac{(a_n+b_n)}{2}}{\sqrt{a_n+b_n}} \stackrel{P}{\to} 0
\Rightarrow  \frac{\hat{d}_n-\frac{(a_n+b_n)}{2}}{a_n+b_n} \stackrel{P}{\to}0
\Rightarrow  \sqrt{\frac{\hat{d}_n}{\frac{a_n+b_n}{2}}} \stackrel{P}{\to} 1.\\
\Rightarrow & \sqrt{\hat{d}_n}- \sqrt{\frac{a_n+b_n}{2}}= o_p(\sqrt{a_n+b_n})=o_p(a_n-b_n)
\end{split}
\end{equation}
So $\sqrt{\hat{d}_n}\hat{f}_{n,k_n}-\frac{a_n-b_n}{2}=o_p(a_n-b_n)$ under $\mathbb{P}_n$. As a consequence, the estimators 
$\hat{A}= \hat{d}_n+\sqrt{\hat{d}_n}\hat{f}_{n,k_n}$ and $\hat{B}=\hat{d}_n-\sqrt{\hat{d}_n}\hat{f}_{n,k_n}$ have the required property. This concludes the proof. 
\hfill $\square$
\section{Proof of non reconstructability}\label{sec:6}
In this section we provide a proof of the non-reconstruction results stated in Theorem \ref{thm:nonrecon}. Our proof technique relies on fine analysis of some conditional probabilities. Technically, this proof is closely related to the non-reconstruction proof in section 6.2 of Banks et al. \cite{bank16} rather than the original proof given in Mossel et al. \cite{MNS12}. At first we prove one Proposition and one Lemma which will be crucial for our proof. 
\begin{proposition}\label{prop:tv}
 Suppose $a_n,b_n \to \infty$, $\frac{a_n}{n} \to p \in [0,1)$ and $c:=\frac{(a_n-b_n)^2}{(a_n+b_n)}<2(1-p)$. Then for any fixed $r$ and any two configurations $(\sigma_{1}^{(1)},\ldots,\sigma_{r}^{(1)})$, $(\sigma_{1}^{(2)},\ldots,\sigma_{r}^{(2)})$
\[
\mathrm{TV}\left(\mathbb{P}_n(G|(\sigma_{1}^{(1)},\ldots,\sigma_{r}^{(1)})),  \mathbb{P}_n(G|(\sigma_{1}^{(2)},\ldots,\sigma_{r}^{(2)})) \right) = o(1)
\]
Here $\mathrm{TV}(\mu_1,\mu_2)$ is the total variation distance between two probability measures $\mu_1$ and $\mu_2$.
\end{proposition}
\begin{proof}
We know that
\begin{equation}\label{eqn:tv}
\begin{split}
&\mathrm{TV}\left(\mathbb{P}_n(G| \sigma_{u}^{(1)}~~ u \in [r]), \mathbb{P}_{n}(G| \sigma_{u}^{(2)}~~ u \in [r])\right)\\
& = \sum_{G} \left|(\mathbb{P}_n(G| \sigma_{u}^{(1)}~~ u \in [r])- \mathbb{P}_{n}(G| \sigma_{u}^{(2)}~~ u \in [r])\right|\\
& = \sum_{G}  \left|(\mathbb{P}_n(G| \sigma_{u}^{(1)}~~ u \in [r])- \mathbb{P}_{n}(G| \sigma_{u}^{(2)}~~ u \in [r])\right|\frac{\sqrt{\mathbb{P}_n'(G)}}{\sqrt{\mathbb{P}_n'(G)}}\\
&\le \left( \sum_{G} \mathbb{P}_n'(G) \right)^{\frac{1}{2}}\left(  \sum_{G}\frac{\left(\mathbb{P}_n(G| \sigma_{u}^{(1)}~~ u \in [r])- \mathbb{P}_{n}(G| \sigma_{u}^{(2)}~~ u \in [r]\right)^2}{\mathbb{P}_n'(G)}\right)^{\frac{1}{2}}\\
&= \left(  \sum_{G}\frac{\left(\sum_{\tilde{\sigma}}\mathbb{P}_n(\tilde{\sigma})\left(\mathbb{P}_n(G| \sigma^{(1)},\tilde{\sigma})- \mathbb{P}_{n}(G| \sigma^{(2)},\tilde{\sigma}\right)\right)^2}{\mathbb{P}_n'(G)}\right)^{\frac{1}{2}}.
\end{split}
\end{equation}
Here $\sigma^{(1)}:= \left\{(\sigma_{1}^{(1)},\ldots,\sigma_{r}^{(1)} \right\}$, $\sigma^{(2)}:= \left\{ (\sigma_{1}^{(2)},\ldots,\sigma_{r}^{(2)}) \right\}$ and $\tilde{\sigma}$ is any configuration on $\{r+1,\ldots,n \}.$ 

Now observe that
\begin{equation}
\begin{split}
&\left(\sum_{\tilde{\sigma}}\mathbb{P}_n(\tilde{\sigma})\left(\mathbb{P}_n(G| \sigma^{(1)},\tilde{\sigma})- \mathbb{P}_{n}(G| \sigma^{(2)},\tilde{\sigma}\right)\right)^2 \\
&= \sum_{\tilde{\sigma},\tilde{\tau}}\mathbb{P}_n(\tilde{\sigma}) \mathbb{P}_{n}(\tilde{\tau})\left(\mathbb{P}_n(G| \sigma^{(1)},\tilde{\sigma})\mathbb{P}_n(G| \sigma^{(1)},\tilde{\tau})+ \mathbb{P}_n(G| \sigma^{(2)},\tilde{\sigma})\mathbb{P}_n(G| \sigma^{(2)},\tilde{\tau})\right.\\
&\left. -  \mathbb{P}_n(G| \sigma^{(1)},\tilde{\sigma})\mathbb{P}_n(G| \sigma^{(2)},\tilde{\tau})- \mathbb{P}_n(G| \sigma^{(2)},\tilde{\sigma})\mathbb{P}_n(G| \sigma^{(1)},\tilde{\tau}) \right).
\end{split}
\end{equation}
We shall prove that the value of 
\begin{equation}\label{eqn:same}
\sum_{G}\sum_{\tilde{\sigma},\tilde{\tau}} \mathbb{P}_n(\tilde{\sigma})\mathbb{P}_n(\tilde{\tau})\frac{ \mathbb{P}_n(G| \sigma^{(1)},\tilde{\sigma})\mathbb{P}_n(G| \sigma^{(2)},\tilde{\tau})}{\mathbb{P}_n'(G)}
\end{equation}
doesn't depend on $\sigma^{(1)}$ and $\sigma^{(2)}$ upto $o(1)$ terms. This will prove that the final expression in (\ref{eqn:tv}) goes to $0$. As a consequence, the proof of Proposition \ref{prop:tv} will be complete. 

At first we recall the definition of $W_{uv}(G,\sigma)$ from (\ref{defn:W}). It is easy to observe that 
\begin{equation}\label{eqn:radonexp}
\begin{split}
 &\sum_{G}\sum_{\tilde{\sigma},\tilde{\tau}}\frac{\mathbb{P}_n(\tilde{\sigma})\mathbb{P}_n(\tilde{\tau})\left( \mathbb{P}_n(G| \sigma^{(1)},\tilde{\sigma})\mathbb{P}_n(G| \sigma^{(2)},\tilde{\tau}) \right)}{\mathbb{P}_n'(G)}\\
 &= \sum_{\tilde{\sigma},\tilde{\tau}}\frac{1}{2^{2(n-r)}}\sum_{G}\left(\prod_{uv}W(G,\sigma^{(1)},\tilde{\sigma}) W(G,\sigma^{(2)},\tilde{\tau})\right)\mathbb{P}_n'(G)\\
 & = \frac{1}{2^{2(n-r)}}\sum_{\tilde{\sigma},\tilde{\tau}} \prod_{u,v}\E_{\mathbb{P}_n'}(W(G,\sigma^{(1)},\tilde{\sigma}) W(G,\sigma^{(2)},\tilde{\tau})).
\end{split}
\end{equation}
Observe that the sum in the final expression of (\ref{eqn:radonexp}) is taken over $(\tilde{\sigma},\tilde{\tau})$ so the configurations in $\sigma^{(1)}$ and $\sigma^{(2)}$ remain unchanged. 

Now let us introduce the following parameters
\begin{equation}
\begin{split}
\rho^{\mathrm{fix}} &:= \frac{1}{r} \sum_{i=1}^{r} \sigma^{(1)}_i \sigma^{(2)}_i\\
S_{\pm}^{\mathrm{fix}} &:= \sum_{u,v \in [r]} I_{\{\sigma^{(1)}_{u}\sigma^{(1)}_{v}\sigma^{(2)}_{u}\sigma^{(2)}_v= \pm 1\}}
\end{split}
\end{equation}
where $I_{A}$ denotes the indicator variable corresponding to set $A$. We similarly define 
\begin{equation}
\begin{split}
\rho(\tilde{\sigma},\tilde{\tau}) &:= \frac{1}{n-r} \sum_{i=r+1}^{n} \tilde{\sigma}_{i} \tilde{\tau}_i\\
S_{\pm}(\tilde{\sigma},\tilde{\tau}) &:= \sum_{u,v \in [r]} I_{\{\tilde{\sigma}_{u}\tilde{\sigma}_{v}\tilde{\tau}_{u}\tilde{\tau}_v= \pm 1\}}.
\end{split}
\end{equation}

By using arguments similar to the proof of Lemma \ref{lem:secondmoment} one can show that the R.S. of the final expression of (\ref{eqn:radonexp}) further simplifies to 
\begin{equation}
\begin{split}
&=\left( 1+ \frac{t_n}{n}\right)^{S_{+}^{\mathrm{fix}}}\left(1-\frac{t_n}{n}\right)^{S_{-}^{\mathrm{fix}}} \frac{1}{2^{2(n-r)}}\sum_{\tilde{\sigma},\tilde{\tau}} \left( 1+ \frac{t_n}{n}\right)^{S_{+}(\tilde{\sigma},\tilde{\tau})}\left(1-\frac{t_n}{n}\right)^{S_{-}(\tilde{\sigma},\tilde{\tau})}\\
&=\left( 1+ \frac{t_n}{n}\right)^{S_{+}^{\mathrm{fix}}}\left(1-\frac{t_n}{n}\right)^{S_{-}^{\mathrm{fix}}} \frac{1}{2^{2(n-r)}}\sum_{\tilde{\sigma},\tilde{\tau}} \left( 1+ \frac{t_n}{n}\right)^{\left(1+\rho(\tilde{\sigma},\tilde{\tau})^2\right)\frac{(n-r)^2}{4}-\frac{n-r}{2}} \left( 1- \frac{t_n}{n}\right)^{\left( 1-\rho(\tilde{\sigma},\tilde{\tau})^2 \right)\frac{(n-r)^2}{4}}.
\end{split}
\end{equation}

Now $S_{+}^{\mathrm{fix}}$ and $S_{-}^{\mathrm{fix}}$ are both bounded by $r^2$ also $t_n=(1+o(1))t$. So 
\[
\left( 1+ \frac{t_n}{n}\right)^{S_{+}^{\mathrm{fix}}}\left(1-\frac{t_n}{n}\right)^{S_{-}^{\mathrm{fix}}}=(1+o(1)).
\]
On the other hand one can repeat the arguments in the proof of Lemma \ref{lem:secondmoment} to conclude that 
\[
\sum_{\tilde{\sigma},\tilde{\tau}} \left( 1+ \frac{t_n}{n}\right)^{\left(1+\rho(\tilde{\sigma},\tilde{\tau})^2\right)\frac{(n-r)^2}{4}-\frac{n-r}{2}} \left( 1- \frac{t_n}{n}\right)^{\left( 1-\rho(\tilde{\sigma},\tilde{\tau})^2 \right)\frac{(n-r)^2}{4}} \to \frac{1}{\sqrt{1-t}} \exp \left\{ -\frac{t}{2}-\frac{t^2}{4} \right\}.
\]
As a result 
\[
\sum_{G}\sum_{\tilde{\sigma},\tilde{\tau}} \mathbb{P}_n(\tilde{\sigma})\mathbb{P}_n(\tilde{\tau})\frac{ \mathbb{P}_n(G| \sigma^{(1)},\tilde{\sigma})\mathbb{P}_n(G| \sigma^{(2)},\tilde{\tau})}{\mathbb{P}_n'(G)}= (1+o(1))\frac{1}{\sqrt{1-t}} \exp \left\{ -\frac{t}{2}-\frac{t^2}{4} \right\}
\]
irrespective of the value of $\sigma^{(1)}$ and $\sigma^{(2)}$. So the final expression in (\ref{eqn:tv}) goes to $0$. Hence the proof is complete.
\end{proof}
We now prove the following easy consequence of Proposition \ref{prop:tv} which  states that the posterior distribution of a single label is essentially unchanged
if we know a bounded number of other labels.
\begin{lemma}\label{lem:posterior}
Suppose $S$ is a set of finite cardinality $r$, $u \notin S$ be a fixed node and $\pi$ gives probability $\frac{1}{2}$ to both $\pm 1$. Then under the conditions of Proposition \ref{prop:tv}
\[
\E\left[ \mathrm{TV}(\mathbb{P}_n(\sigma_u|G,\sigma_{S}),\pi)|\sigma_{S} \right]=o(1).
\]   
\end{lemma}
\begin{proof}
Observe that $\mathbb{P}_n(\sigma_u=i)= \pi(i)$ from the model assumption. So 
\begin{equation}\label{eqn:posterior}
\begin{split}
\E\left[ \mathrm{TV}(\mathbb{P}_n(\sigma_u|G,\sigma_{S}),\pi)|\sigma_{S} \right] &= \sum_{G}\sum_{i=\pm 1}\left|\mathbb{P}_n\left(\sigma_u=i|G,\sigma_S\right)-\mathbb{P}_n (\sigma_u=i) \right|\mathbb{P}_n(G|\sigma_S)\\
&= \sum_{i= \pm 1}\mathbb{P}_n(\sigma_u=i) \sum_{G}\left|\frac{\mathbb{P}_n\left(\sigma_u=i|G,\sigma_S\right)}{\mathbb{P}_n(\sigma_u=i)}-1  \right|\mathbb{P}_n(G|\sigma_S)\\
&= \sum_{i=\pm 1}\mathbb{P}_n(\sigma_u=i)\sum_{G} \left|\frac{\mathbb{P}_n\left(\sigma_u=i\cap G \cap \sigma_S \right) \mathbb{P}_n(\sigma_S) }{\mathbb{P}_n(\sigma_u \cap \sigma_S)\mathbb{P}_n(G\cap \sigma_S)}-1  \right|\mathbb{P}_n(G|\sigma_S)\\
&= \sum_{i=\pm 1}\mathbb{P}_n(\sigma_u=i)\sum_{G} \left|\frac{\mathbb{P}_n(G|\sigma_S,\sigma_u=i)}{\mathbb{P}_n(G|\sigma_S)} -1 \right|\mathbb{P}_n(G|\sigma_S)\\
\end{split}
\end{equation}
Observe that 
\[
\mathbb{P}_n(G|\sigma_S)= \frac{1}{2}\left(\mathbb{P}_n(G|\sigma_S,\sigma_u=1)+ \mathbb{P}_n(G|\sigma_S,\sigma_u=-1)\right).
\]
As a consequence, the final expression of the R.S. of (\ref{eqn:posterior}) becomes
\[
\frac{1}{2}\sum_{i=\pm 1}\mathbb{P}_n(\sigma_u=i)\mathrm{TV}\left( \mathbb{P}_n(G|\sigma_S,\sigma_u=i),\mathbb{P}_n(G|\sigma_S,\sigma_u=-i) \right).
\]
So the proof is complete by applying Proposition \ref{prop:tv}. 
\end{proof}

\noindent
With Proposition \ref{prop:tv} and Lemma \ref{lem:posterior} in hand, we now give a proof of Theorem \ref{thm:nonrecon}.

\noindent
\textbf{Proof of Theorem \ref{thm:nonrecon}:}
We only prove part $i)$ of Theorem \ref{thm:nonrecon}. The proof of part $ii)$ is similar.

\noindent
Let $\hat{\sigma}$ be any estimate of the labeling of the nodes, $\sigma$ be the true labeling and $f:\{1,2\}\to \{\pm 1\}$ be the function such that $f(1)=1$ and $f(2)=-1$.

\noindent
It is elementary to check that 
\begin{equation}
\begin{split}
\frac{1}{2}\mathrm{ov}(\sigma,\hat{\sigma})= \frac{1}{n}\left[N_{11}+N_{22}-\frac{1}{n}(N_{1\cdot}N_{\cdot 1}) -\frac{1	}{n}(N_{2\cdot}N_{\cdot 2})\right].
\end{split}
\end{equation}
Here 
\begin{equation}
\begin{split}
N_{ij}&= \left|\sigma^{-1}\{ f(i)\} \cap \hat{\sigma}^{-1}\{ f(j) \} \right|\\
N_{i\cdot}&= \left|\sigma^{-1}\{ f(i)\}  \right|\\
N_{\cdot j} &= \left| \hat{\sigma}^{-1}\{ f(j) \} \right|. 
\end{split}
\end{equation} 
So it is sufficient to prove that 
\[
\frac{1}{n^2} \E_{\mathbb{P}_n}\left[ N_{ii}- \frac{1}{n} N_{i\cdot} N_{\cdot i} \right]^2= \frac{1}{n^2} \E_{\mathbb{P}_n}\left[ N_{ii}^2- \frac{2}{n}N_{ii} N_{i\cdot} N_{\cdot i} + \frac{1}{n^2} N_{i\cdot}^2 N_{\cdot i}^2  \right] \to 0 ~~ i\in \{ 1,2 \}.
\]
Now 
\begin{equation}\label{expectN_ii}
\begin{split}
\E_{\mathbb{P}_n}\left[ N_{ii}^2\right]&= \E_{\mathbb{P}_n} \left[ \sum_{u, v} I_{\{\sigma_{u}= f(i)\}}I_{\{\sigma_{v}= f(i)\}}I_{\{\hat{\sigma}_{u}= f(i)\}} I_{\{\hat{\sigma}_{v}= f(i)\}} \right]\\
&= \E_{\mathbb{P}_n}\left[ \E \left[ \sum_{u, v} I_{\{\sigma_{u}= f(i)\}}I_{\{\sigma_{v}= f(i)\}}I_{\{\hat{\sigma}_{u}= f(i)\}} I_{\{\hat{\sigma}_{v}= f(i)\}} \right]\left|G\right.\right]\\
&= \E_{\mathbb{P}_n}\left[ \E \left[ \sum_{u, v} I_{\{\sigma_{u}= f(i)\}}I_{\{\sigma_{v}= f(i)\}}\right]I_{\{\hat{\sigma}_{u}= f(i)\}} I_{\{\hat{\sigma}_{v}= f(i)\}}\left|G\right.\right]
\end{split}
\end{equation}
The last step follows from the fact that $\hat{\sigma}$ is a function of $G$. Now 
\begin{equation*}
\begin{split}
\E\left[ I_{\{\sigma_{u}= f(i)\}}I_{\{\sigma_{v}= f(i)\}}|G\right]&= \E\left[ I_{\{\sigma_{u}= f(i)\}}|G,\sigma_{v}= f(i)\right]\mathbb{P}_{n}\left( \sigma_{v}=f(i)|G \right)\\
&= (\pi(f(i))+o(1)) \mathbb{P}_{n}(G|\sigma_{v}=f(i))\frac{\mathbb{P}_n(\sigma_v=f(i))}{\mathbb{P}_n(G)}\\
&= (\pi^{2}(f(i))+o(1))\frac{\mathbb{P}_{n}(G|\sigma_{v}=f(i))}{\mathbb{P}_n(G)}
\end{split}
\end{equation*}
Here the second step follows from Lemma \ref{lem:posterior}.
As a consequence,
\begin{equation}
\begin{split}
&\left|\E_{\mathbb{P}_n}\left[ \E  \sum_{u, v} \left(I_{\{\sigma_{u}= f(i)\}}I_{\{\sigma_{v}= f(i)\}}-\pi^{2}(f(i))\right)I_{\{\hat{\sigma}_{u}= f(i)\}} I_{\{\hat{\sigma}_{v}= f(i)\}}\left|G\right.\right]\right|\\
&\le \E_{\mathbb{P}_n}\left[ \sum_{u,v} \left| \E\left[\left(I_{\{\sigma_{u}= f(i)\}}I_{\{\sigma_{v}= f(i)\}}-\pi^{2}(f(i))\right)I_{\{\hat{\sigma}_{u}= f(i)\}} I_{\{\hat{\sigma}_{v}= f(i)\}}\left|G\right. \right]\right| \right]\\
& = \E_{\mathbb{P}_n} \left[ \sum_{u,v} \left| \pi^{2}(f(i))I_{\{\hat{\sigma}_{u}= f(i)\}} I_{\{\hat{\sigma}_{v}= f(i)\}}\left( \frac{\mathbb{P}_{n}(G|\sigma_{v}=f(i))}{\mathbb{P}_n(G)} -1 \right) +o(1) \right| \right]\\
& \le \sum_{u,v}\sum_{G}\left|\mathbb{P}_{n}(G|\sigma_{v}=f(i))-\mathbb{P}_n(G)\right|+o(n^2)\\
&= o(n^2).
\end{split}
\end{equation}
Here the last step follows from Proposition \ref{prop:tv}.
\\
So we have 
\begin{equation}\label{e1}
\E_{\mathbb{P}_n}\left[ N_{ii}^2\right]= \sum_{u,v}\E_{\mathbb{P}_n} \left[ \pi^2(f(i))I_{\{\hat{\sigma}_{u}= f(i)\}} I_{\{\hat{\sigma}_{v}= f(i)\}} \right] +o(n^2)
\end{equation}
Similar calculations will prove that
\begin{equation}\label{e2}
\E_{\mathbb{P}_n}\left[ N_{ii} N_{i\cdot} N_{\cdot i}\right]= n \sum_{u,v}\E_{\mathbb{P}_n} \left[ \pi^2(f(i))I_{\{\hat{\sigma}_{u}= f(i)\}} I_{\{\hat{\sigma}_{v}= f(i)\}} \right] +o(n^3)
\end{equation}
and 
\begin{equation}\label{e3}
\E_{\mathbb{P}_n}\left[N_{i \cdot}^2 N_{\cdot i}^2\right]= n^2\sum_{u,v}\E_{\mathbb{P}_n} \left[ \pi^2(f(i))I_{\{\hat{\sigma}_{u}= f(i)\}} I_{\{\hat{\sigma}_{v}= f(i)\}} \right] +o(n^4).
\end{equation}
Plugging in these estimates we have 
\[
\frac{1}{n^2} \E_{\mathbb{P}_n}\left[ N_{ii}- \frac{1}{n} N_{i\cdot} N_{\cdot i} \right]^2=o(1).
\] 
This completes the proof.
\hfill{$\square$}
\bibliography{CON_SBM}
\section{Appendix}
\subsection{More general words and their equivalence classes}
Here we only give a very brief description about the combinatorial aspects of random matrix theory required to prove Lemma \ref{lem:appendix}. For more general information one should look at Chapter 1 of Anderson et al. \cite{AGZ} and Anderson and Zeiouni \cite{AZ05}. The definitions in this section have been taken from  Anderson et al. \cite{AGZ} and Anderson and Zeitouni \citep{AZ05}.   
\begin{definition}($\mathcal{S}$ words)
Given a set $\mathcal{S}$, an $\mathcal{S}$ letter $s$ is simply an element of
$\mathcal{S}$. An $\mathcal{S}$ word $w$ is a finite sequence of letters $s_1
\ldots s_n$, at least one letter long.
An $\mathcal{S}$ word $w$ is closed if its first and last letters are the same. Two $\mathcal{S}$ words
$w_1,w_2$ are called equivalent, denoted $w_1\sim w_2$, if there is a bijection on $\mathcal{S}$ that
maps one into the other.
\end{definition}
When $\mathcal{S} = \{1, \ldots ,N\}$ for some finite $N$, we use the term $N$ word. Otherwise, if
the set $\mathcal{S}$ is clear from the context, we refer to an $\mathcal{S}$ word simply as a word.

For any word $w = s_1 \ldots s_k$, we use $l(w) = k$ to denote the length of $w$, define
the weight $wt(w)$ as the number of distinct elements of the set ${s_1,\ldots , s_k
}$ and the
support of $w$, denoted by $\mathrm{supp}(w)$, as the set of letters appearing in $w$. With any word
$w$ we may associate an undirected graph, with $wt(w)$ vertices and $l(w) −- 1$ edges,
as follows.
\begin{definition} (Graph associated with a word) 
Given a word $w = s_1 \ldots s_k$,
we let $G_w = (V_w,E_w)$ be the graph with set of vertices $V_w = \mathrm{supp}(w)$ and (undirected)
edges $E_w = \{\{s_i, s_i+1
\}, i = 1,\ldots ,k - 1
\}.$
\end{definition}
The graph $G_w$ is connected since the word $w$ defines a path connecting all the
vertices of $G_w$, which further starts and terminates at the same vertex if the word
is closed. For $e \in E_w$, we use $N^w_
e$ to denote the number of times this path traverses
the edge $e$ (in any direction). We note that equivalent words generate the same
graphs $G_w$ (up to graph isomorphism) and the same passage-counts $N^w_
e$.
\begin{definition}(sentences and corresponding graphs)
A sentence $a=[w_i]_{i=1}^{n}=[[\alpha_{i,j}]_{j=1}^{l(w_i)}]_{i=1}^{n}$ is an ordered collection of $n$ words of length $(l(w_1),\ldots,l(w_n))$ respectively. We define the graph $G_a=(V_a,E_a)$ to be the graph with 
\[
V_a= \mathrm{supp}(a), E_a= \left\{ \{ \alpha_{i,j},\alpha_{i,j+1}\}| i=1,\ldots,n ; j=1,\ldots, l(w_i)-1 \}  \right\}. 
\] 
\end{definition} 
\begin{definition}(weak CLT sentences)
A sentence $a=[w_i]_{i=1}^{n}$ is called a weak CLT sentence. If the following conditions are true:
\begin{enumerate}
\item All the words $w_i$'s  are closed.
\item Jointly the words $w_i$ visit edge of $G_a$ at least twice.
\item For each $i\in \{1,\ldots,n \}$, there is another $j\neq i \in \{ 1,\ldots,n\}$ such that $G_{w_i}$ and $G_{w_j}$ have at least one edge in common. 
\end{enumerate}
\end{definition}
Note that these definitions are consistent with the ones given in Section \ref{section:4}. However, in Section \ref{section:4}, we defined these only for some specific cases required to solve the problem.

In order to prove Lemma \ref{lem:appendix}, we require the following result from Anderson et al. \cite{AGZ}.
\begin{lemma}(Lemma 2.1.23 in Anderson et al. \cite{AGZ})\label{lem:bounded}
Let $\mathcal{W}_{k,t}$ denote the equivalence classes corresponding to all closed words $w$ of length $k+1$ with $\mathrm{wt}(w)=t$ such that each edge in $G_w$ have been traversed at least twice. Then for $k>2t-2$,
\[
\# \mathcal{W}_{k,t} \le 2^k k^{3(k-2t+2)}
\] 
\end{lemma}

\noindent
Assuming Lemma \ref{lem:bounded} we now prove Lemma \ref{lem:appendix}.\\
\textbf{Proof of Lemma \ref{lem:appendix}:}
Let $a=[w_{i}]_{i=1}^{m}$ be a  weak CLT sentence such that $G_a$ have $\mathcal{C}(a)$ many connected components. At first we introduce a partition $\eta(a)$ in the following way. We put $i$ and $j$ in same block of $\eta(a)$ if $G_{w_i}$ and $G_{w_j}$ share an edge. At first we fix such a partition $\eta$ and consider all the sentences such that $\eta(a)=\eta$.  Let $\mathcal{C}(\eta)$ be the number of blocks in $\eta$. It is easy to observe that for any $a$ with $\eta(a)=\eta$, we have $\mathcal{C}(\eta)=\mathcal{C}(a)$. From now on we denote $\mathcal{C}(\eta)$ by $\mathcal{C}$ for convenience.
\\
 Let $a$ be any weak CLT sentence such that $\eta(a)=\eta$. We now propose an algorithm to embed $a$ into $\mathcal{C}$ ordered closed words $(W_1,\ldots,W_{\mathcal{C}})$ such that the equivalence class of each $W_i$ belongs to $\mathcal{W}_{L_i,t_i}$ for some numbers $L_i$ and $t_i$.
 
A similar type of argument can be found in Claim 3 of the proof of Theorem 2.2 in Banerjee and Bose(2016) \cite{BanB16}.
\\
\textbf{An embedding algorithm:} 
Let $B_1,\ldots,B_\mathcal{C}$ be the blocks of the partition $\eta$ ordered in the following way. Let $m_i= \min\{j: j \in B_i\}$ and we order the blocks $B_i$ such that $m_1<m_2\ldots < m_{\mathcal{C}}$. Given a partition $\eta$ this ordering is unique.  Let 
\[
B_i=\{ i(1)<i(2)<\ldots<i(l(B_i)) \}.
\] 
Here $l(B_i)$ denotes the number of elements in $B_i$. 

For each $B_i$ we embed the sentence $a_i=[w_{i(j)}]_{1\le j \le l(B_i) }$ into $W_i$ sequentially in the following manner.
\begin{enumerate}
\item Let $S_1=\{i(1)\}$ and $\mathfrak{w}_1=w_{i(1)}.$
\item For each $1\le c \le l(B_i)-1$ we perform the following.
\begin{itemize}
\item Consider $\mathfrak{w}_c=(\alpha_{1,c},\ldots,\alpha_{l(\mathfrak{w}_c),c})$ and $S_c \subset B_{i}$. Let $ne \in B_{i} \backslash S_c$ be the index such that the following  two conditions hold.
\begin{enumerate}
\item $G_{\mathfrak{w}_c}$ and $G_{w_{ne}}$ shares at least one edge $e=\{\alpha_{\kappa_1,c},\alpha_{\kappa_1+1,c}\}$.
\item $\kappa_1$ is minimum among all such choices.
\end{enumerate}
\item Let $w_{ne}=(\beta_{1,c},\ldots, \beta_{l(w_{ne}),c})$ and $\{\beta_{\kappa_2,c},\beta_{\kappa_2+1,c}\}$ be the first time $e$ appears in $w_{ne}$. As $\{\beta_{\kappa_2,c},\beta_{\kappa_2+1,c}\}=\{\alpha_{\kappa_1,c},\alpha_{\kappa_1+1,c}\}$, $\alpha_{\kappa_1,c}$ is either equal to $\beta_{\kappa_2,c}$ or $\beta_{\kappa_2,c}$. Let $\kappa_3 \in \{ \kappa_2,\kappa_2+1 \}$ such that $\alpha_{\kappa_1,c}=\beta_{\kappa_3,c}$. If $\beta_{\kappa_2,c}=\beta_{\kappa_2+1,c}$, then we simply take $\kappa_3=\kappa_2$.
\item We now generate $\mathfrak{w}_{c+1}$ in the following way 
$$\mathfrak{w}_{c+1}=(\alpha_{1,c},\ldots,\alpha_{\kappa_1,c},\beta_{\kappa_3+1,c},\ldots, \beta_{l(w_{ne}),c},\beta_{2,c},\ldots, \beta_{\kappa_3,c},\alpha_{\kappa_1+1,c},\ldots, \alpha_{l(\mathfrak{w}_{c}),c} ).$$

Let $\tilde{a}_c:=(\mathfrak{w}_c,w_{ne})$. It is easy to observe by induction that all $\mathfrak{w}_{c}$'s are closed words and so are all the $w_{ne}$'s. So the all the edges  in the graph $G_{\tilde{a}_c}$ are preserved along with their passage counts in $G_{\mathfrak{w}_{c+1}}$.
\item Generate $S_{c+1}=S_{c}\cup \{ ne \}.$
\end{itemize}
\item Return $W_{i}=\mathfrak{w}_{l(B_i)}$.
\end{enumerate}
In the preceding algorithm we have actually defined a function $f$ which maps any weak CLT sentence $a$ into $\mathcal{C}$ ordered closed words $(W_1,\ldots,W_{\mathcal{C}})$ such that each the equivalence class of each $W_i$ belongs to $\mathcal{W}_{L_i,t_i}$ for some numbers $L_i$ and $t_i$. Observe also that $L_i< \sum_{j \in B_i} l(w_j)$ and $t_i < \frac{L_i+1}{2}$.
\vspace{10pt}

\noindent
Unfortunately $f$ is not an injective map. So given $(W_1,\ldots,W_{\mathcal{C}})$ we find an upper bound to the cardinality of the following set 
\[
f^{-1}(W_1,\ldots,W_{\mathcal{C}}):=\{a| f(a)=(W_1,\ldots,W_{\mathcal{C}})\}  
\]
We have argued earlier $\mathcal{C}$ is the number of blocks in $\eta$.
However, in general $(W_1,\ldots,W_{\mathcal{C}})$ does neither specify the partition $\eta$ nor the order in which the words are concatenated with in each block $B_i$ of $\eta$. So we fix a partition $\eta$ with $\mathcal{C}$ many blocks and an order of concatenation $\mathcal{O}$. Observe that 
\[
\mathcal{O}= (\sigma_1(\eta),\ldots,\sigma_{\mathcal{C}}(\eta))
\]
where for each $i$, $\sigma_{i}(\eta)$ is a permutation of the elements in $B_i$. Now we give an uniform upper bound to the cardinality of the following set 
\[
f^{-1}_{\eta,\mathcal{O}}(W_1,\ldots,W_{\mathcal{C}}):=\left\{a| \eta(a)=\eta ~~; \mathcal{O}(a)= \mathcal{O} ~~ \& f(a)=(W_1,\ldots,W_{\mathcal{C}}) \right\}.
\]
According to the algorithm any word $W_i$ is formed by recursively applying step 2. to $(\mathfrak{w}_{c},w_{ne})$ for $1 \le c \le l(B_i)$. 
Given a word $\mathfrak{w}_3=(\alpha_1,\ldots,\alpha_{l(\mathfrak{w}_3)})$, we want to find out the number of two words sentences $(\mathfrak{w}_1,\mathfrak{w}_2)$ such that applying step 2 of the algorithm on $(\mathfrak{w}_1,\mathfrak{w}_2)$ gives $\mathfrak{w}_3$ as an output. This is equivalent to choose three positions $i_1<i_2<i_3$ from the set $\{1,\ldots,l(\mathfrak{w}_3)\}$ such that $\alpha_{i_1}=\alpha_{i_3}$. Once these three positions are chosen, $(\mathfrak{w}_1,\mathfrak{w}_2)$ can be constructed uniquely in the following manner 
\begin{equation*}
\begin{split}
\mathfrak{w}_1&=(\alpha_1,\ldots, \alpha_{i_1},\alpha_{i_3+1},\ldots,\alpha_{l(\mathfrak{w}_3)})\\
\mathfrak{w}_2&=(\alpha_{i_2},\ldots,\alpha_{i_3},\alpha_{i_1+1},\ldots,\alpha_{i_2}).
\end{split}
\end{equation*}

Total number of choices $i_1<i_2<i_3$ is bounded by $l(\mathfrak{w}_3)^{3}\le \left(\sum_{i=1}^{m}{l(w_i})\right)^{3}.$ For each block $B_i$, step 2. of the algorithm has been used $l(B_i)$ many times. So 
\[
f^{-1}_{\eta,\mathcal{O}}(W_1,\ldots,W_{\mathcal{C}}) \le \left(\sum_{i=1}^{m}{l(w_i})\right)^{3\sum_{i=1}^{\mathcal{C}}l(B_i)}= \left(\sum_{i=1}^{m}{l(w_i})\right)^{3m}
\]
On the other hand, a there at most $m^{m}$ many $\eta$'s and for each $\eta$ there are at most $\prod_{i=1}^{\mathcal{C}}l(B_i)!\le m^{m}$ choices of $\mathcal{O}$. So
\begin{equation}\label{bdd1}
f^{-1}(W_1,\ldots,W_{\mathcal{C}})\le m^{2m}\left(\sum_{i=1}^{m}{l(w_i})\right)^{3m}\le \left(D_1\sum_{i=1}^{m}{l(w_i})\right)^{D_2m}
\end{equation}
for some known constants $D_1$ and $D_2$. Now we fix the sequence $(L_i,t_i)$ and find an upper bound to the number of $(W_1,\ldots, W_{\mathcal{C}})$. From Lemma \ref{lem:bounded} we know the number of choices of $W_i$ is bounded by $2^{L_i-1}(L_i-1)^{L_i-2t+1}n^{t_i}$. So the total number of choices for $(W_1,\ldots, W_{\mathcal{C}})$ is bounded by 
\begin{equation}\label{bdd2}
2^{\sum_{i=1}^{m}l(w_i)}\prod_{i=1}^{\mathcal{C}}(L_i-1)^{3(L_i-2t+1)}n^{t_i}\le 2^{\sum_{i=1}^{m}l(w_i)}n^{t}\left(\sum_{i=1}^{m}l(w_i)\right)^{3(\sum_{i=1}^{m}l(w_i)-2t)}\left(\sum_{i=1}^{m}l(w_i)\right)^{m}.
\end{equation}
Now the number of choices $(L_i,t_i)$ such that $\sum_{i=1}^{\mathcal{C}}L_{i}= \sum_{i=1}^{m}l(w_i)$ and $\sum_{i=1}^{\mathcal{C}}t_i= t$ are bounded by 
\begin{equation}\label{bdd3}
\binom{\sum_{i=1}^{m}l(w_i)-1}{\mathcal{C}-1} \binom{t-1}{\mathcal{C}-1}\le  \left(\sum_{i=1}^{m} l(w_i)\right)^{2m}.
\end{equation}
Here the inequality follows since $\mathcal{C}\le m$ and $t\le \sum_{i=1}^{m}\frac{l(w_i)}{2}-1$. Finally we using the fact that $1\le \mathcal{C}\le m$ and combining (\ref{bdd1}), (\ref{bdd2}) and (\ref{bdd3}) we finally have
\begin{equation}
\begin{split}
&\# \mathcal{A} \le \left(D_1\sum_{i=1}^{m}{l(w_i})\right)^{D_2m} \times  2^{\sum_{i=1}^{m}l(w_i)}n^{t}\left(\sum_{i=1}^{m}l(w_i)\right)^{3(\sum_{i=1}^{m}l(w_i)-2t)}\left(\sum_{i=1}^{m}l(w_i)\right)^{m} \times m\left(\sum_{i=1}^{\mathcal{C}} l(w_i)\right)^{2m}\\
\Rightarrow & \#\mathcal{A} \le 2^{\sum_{i} l(w_i)}\left(C_1\sum_{i}l(w_i)\right)^{C_2m}\left(\sum_{i}l(w_i)\right)^{3(\sum_{i}l(w_i)-2t)}n^{t} 
\end{split}
\end{equation}
as required. \hfill $\square$
\\
\textbf{Acknowledgments} The author thanks Elchanan Mossel and Zongming Ma for many useful discussions and their careful reading of the draft. He is grateful to Joe Neeman for useful discussions about non-reconstruction, Jian Ding for pointing out a small mistake in an earlier version of the draft and Cris Moore for pointing out an interesting reference. Finally he thanks Adam Smith and Audra McMillan for their interest in this work and several useful discussions. 
\end{document}